\documentclass[12pt]{amsart}
\usepackage{amssymb}
\usepackage{verbatim}
\usepackage{mathrsfs}
\usepackage{mathtools}

\usepackage{xcolor}

\newtheorem{thm}{Theorem}[section]
\newtheorem{prop}[thm]{Proposition}
\newtheorem{lem}[thm]{Lemma}
\newtheorem{cor}[thm]{Corollary}

\newtheorem{rema}[thm]{Remark}

\theoremstyle{definition}

\theoremstyle{remark}

\newtheorem*{rem}{Remark}

\mathtoolsset{showonlyrefs}

\numberwithin{equation}{section}

\newcommand{\Zag}{\mathscr{L}}
\newcommand{\M}{\mathcal{M}}

\newcommand{\TM}{\vartheta}
\newcommand{\Res}{\mathcal{R}}
\newcommand{\STM}{\Theta}
\newcommand{\W}{\mathcal{W}}
\newcommand{\alt}{\alpha}
\newcommand{\Lo}{ l_1 }
\newcommand{\La}{ h_1 }
\newcommand{\LA}{ W }
\newcommand{\Le}{ h_1 }
\newcommand{\LE}{ W }

\newcommand{\Fs}{ F_2 }
\newcommand{\Lt}{ l_2 }
\newcommand{\Aa}{\mathbf{a}}
\newcommand{\sgh}{ \varepsilon }
\newcommand{\LB}{U}

\newcommand{\rr}{\mathbf{r}}
\newcommand{\HyG}{ {}_2F_1 }

\newcommand{\Mod}[1]{\ (\textup{mod}\ #1)}

\providecommand{\sym}{\operatorname{sym}}

\mathtoolsset{showonlyrefs}

\begin{document}

\title[Hybrid subconvexity for symmetric-square $L$-functions]{Hybrid subconvexity for Maass form symmetric-square $L$-functions}

\author[O.Balkanova]{Olga Balkanova}
\address{
Steklov Mathematical Institute of Russian Academy of Sciences, 8 Gubkina st., Moscow, 119991, Russia}
\email{balkanova@mi-ras.ru}

\author[D. Frolenkov]{Dmitry  Frolenkov}
\address{
Steklov Mathematical Institute of Russian Academy of Sciences, 8 Gubkina st., Moscow, 119991, Russia}
\email{frolenkov@mi-ras.ru}
\thanks{}

\begin{abstract}
Recently R. Khan and M. Young proved a mean Lindel\"{o}f estimate for the second moment of Maass form symmetric-square $L$-functions $L(\sym^2 u_{j},1/2+it)$  on the short interval of length $G\gg |t_j|^{1+\epsilon}/t^{2/3}$, where $t_j$ is a spectral parameter of the corresponding Maass form. Their estimate yields a subconvexity estimate for $L(\sym^2 u_{j},1/2+it)$ as long as $|t_j|^{6/7+\delta}
\ll t<(2-\delta)|t_j|$.  We  obtain a mean Lindel\"{o}f estimate for the same moment in shorter intervals, namely for $G\gg |t_j|^{1+\epsilon}/t$. As a corollary, we prove a subconvexity estimate for $L(\sym^2 u_{j},1/2+it)$  on the interval $|t_j|^{2/3+\delta}\ll t\ll |t_j|^{6/7-\delta}$.
\end{abstract}

\keywords{$L$-functions, moments, Maass forms}
\subjclass[2010]{Primary: 11F11, 11M99, 33C20}

\maketitle


\section{Introduction}
Let $u_j$ be a Hecke-Maass cusp form with Laplace eigenvalue $\kappa_{j}=1/4+t_{j}^2$.  Khan and Young proved the following result (see \cite[Theorem 1.1]{KhYoung}).
\begin{thm}\label{KY t-result}
Let $0<\delta<2$ be fixed, and let $t$, $T$, $G>1$ be such that
\begin{equation}\label{KY t restriction}
\frac{T^{3/2+\delta}}{G^{3/2}}\le t\le(2-\delta)T.
\end{equation}
The following estimate holds:
\begin{equation}\label{t-mean Lindelof}
\sum_{T<t_j\le T+G}\alpha_{j}|L(\sym^2 u_{j},1/2+it)|^2\ll T^{1+\epsilon}G.
\end{equation}
\end{thm}
A direct corollary of Theorem \ref{KY t-result} is the following subconvexity estimate, see \cite[Corollary 1.2]{KhYoung}.
\begin{cor}\label{KY subcon cor}
For $0<\delta<2$  and
\begin{equation}\label{KY t subcon restriction}
|t_j|^{6/7+\delta}<t<(2-\delta)|t_j|
\end{equation}
one has
\begin{equation}\label{t-subcon}
|L(\sym^2 u_{j},1/2+it)|\ll |t_j|^{1+\epsilon}t^{-1/3}.
\end{equation}
\end{cor}
The condition $|t_j|^{6/7+\delta}<t$ guarantees that the estimate \eqref{t-subcon} is better than the convexity estimate
\begin{equation}\label{sym2 convexity}
|L(\sym^2 u_{j},1/2+2it)|\ll (1+|t|)^{1/4+\epsilon}|t_j|^{1/2+\epsilon}.
\end{equation}

Our main result allows us to extend the range \eqref{KY t restriction} in Theorem \ref{KY t-result} as follows.

\begin{thm}\label{BF t-result}
For
$G\gg \max\left(\frac{t^{2/3}}{T^{1/3}},\frac{T}{t}\right)T^{\epsilon}$  and $ t\ll T^{6/7-\epsilon}$
one has
\begin{equation}\label{BF t-mean Lindelof}
\sum_{T<t_j\le T+G}\alpha_{j}|L(\sym^2 u_{j},1/2+it)|^2\ll T^{1+\epsilon}G.
\end{equation}
\end{thm}
\begin{cor}\label{BF subcon cor}
The following bounds hold:
\begin{equation}\label{BF t-subcon}
|L(\sym^2 u_{j},1/2+it)|\ll\frac{|t_j|^{1+\epsilon}}{\sqrt{t}}\quad\hbox{for}\quad|t_j|^{2/3+\epsilon}\ll t\leq |t_j|^{4/5},
\end{equation}
and
\begin{equation}\label{BF t-subcon2}
|L(\sym^2 u_{j},1/2+it)|\ll(|t_j|t)^{1/3}|t_j|^{\epsilon}\quad\hbox{for}\quad|t_j|^{4/5}\leq t\ll |t_j|^{6/7-\epsilon}.
\end{equation}
\end{cor}
The condition $t\gg |t_j|^{2/3+\epsilon}$ guarantees that \eqref{BF t-subcon} is better than the convexity bound \eqref{sym2 convexity}.
The restriction $t\ll |t_j|^{6/7-\epsilon}$ is of technical nature (see \eqref{Q0 conditions}, \eqref{Q0 def} and  \eqref{Q0 conditions GtT}). Though we do not think that this estimate is a limit of our  method (for proving \eqref{BF t-mean Lindelof}), it seems that our approach does not work for $t$ close to $t_j$ as in \eqref{KY t subcon restriction}. 

Note that in the special case of central values $L(\sym^2 u_{j},1/2)$, Khan and Young (see \cite[Theorem 1.3]{KhYoung}) proved the mean Lindel\"{o}f estimate over a much shorter interval, namely for $G\gg T^{1/5}$. In this case, our method yields the same result, as shown in \cite{BFsym2}. The proof of Theorem \ref{BF t-result} is similar to \cite[Theorem 1.3]{KhYoung} in \cite{BFsym2}. The main difference is that the analysis of special functions is technically  more complicated. In fact, a deeper study of hypergeometric functions is exactly what allows us to obtain an improvement.
Since a detailed comparison of our method with the one of Khan and Young has been provided in \cite[Sec.1]{BFsym2} we now discuss only the main steps of the proof, along with the ideas behind them and the reasons why they are beneficial.

Our proof combines a reciprocity type formula for the  twisted first moment with an approximate functional equation. Let us mention that the approach based on reciprocity type formulas has two main advantages. First, it usually simplifies the analysis of the arithmetic part of the proof, since instead of considering multiple Dirichlet series we get some $L$-functions. Of course, some calculations performed on multiple Dirichlet series are included in the calculations leading to $L$-functions, but they are usually simpler.  Second, in the analytic part instead of multiple integrals of exponential functions, Gamma factors and some smooth functions, we obtain either special functions (usually hypergeometric functions) or integrals of special functions. This allows us to use various integral representations (not only Mellin-Barnes integrals) and apply more techniques from the theory of asymptotic analysis of special functions.
For example, in order to obtain an asymptotic expansion of a special function, it is possible to use not only the saddle point method, but also the Liouville-Green method, whereas in the case of multiple integrals only the first method is applicable. 

For these reasons, we would like to study the second moment \eqref{BF t-mean Lindelof} using some formula of this kind. Unfortunately, the existing reciprocity type formula \cite[Theorem 1.1]{BFsym2Explicit} does not isolate the main term of the second moment, which makes it difficult to apply for asymptotic analysis. In order to avoid this problem, we combine both of methods by representing one of $L(\sym^2 u_{j},1/2+it)$ using an approximate functional equation, and then applying the reciprocity type formula for the twisted first moment proved in \cite{Bal} (see also \cite[sec.3]{BF2mom}).  This formula relates the first twisted moment of $L(\sym^2 u_{j},1/2+it)$ to the first moment of Zagier $L$-series. Accordingly, we translate the problem of estimating the second moment \eqref{BF t-mean Lindelof} to the study of the double sum
\begin{equation}\label{eq:doublesum}
\sum_{m\ll T^{1+\epsilon}\sqrt{t}}\left(\sum_{n<2m}+\sum_{n>2m}\right)\Zag_{n^2-4m^2}(1/2+it)\frac{I(n/m,1/2+it)}{n^{1/2-2it}m^{1/2\pm2it}},
\end{equation}
where $I(x,\rho)$ is an integral (over $r$) of the Gauss hypergeometric function (the reason why we distinguish the cases $n<2m$ and $n>2m$ is that the  function $I(x,\rho)$ has different integral representations in these ranges). 

As the next step, we  prove an asymptotic formula for the function $I(x,\rho)$.  To this end, we first substitute asymptotic expansions of the Gauss hypergeometric functions (proved in \cite{BF2F1 I}, \cite{BF2F1 II}) into the integral representation of  $I(x,\rho)$. The proofs of these two asymptotic expansions are quite nontrivial due to the fact that we are looking for a uniform expansion in three variables $x,r,t$. As already mentioned, there are two methods (uniform saddle point method and the Liouville-Green method) for obtaining such expressions. In turns out that in one case it is easier to apply a version of the saddle point method (due to Temme),  and in the second case the Liouville-Green method is more suitable. Note that in the latter case, we perform a two-step process to obtain the asymptotic expansion. At the first stage, the Gauss hypergeometric function is approximated by Bessel functions of large purely imaginary order, and at the second stage, the resulting Bessel functions are approximated by Airy functions. The final result is quite complicated because the argument of the Airy function differs for $x\gtrless1-t^2/r^2$. 
Since such methods are a little unusual for number theorists, we decided to present the results on approximating special functions in separate articles \cite{BF2F1 I} and \cite{BF2F1 II}.
We believe that these asymptotic results are the key point where the improvement of Khan and Young's method was achieved. Indeed, in \cite[Remark before sec.9]{KhYoung} Khan and Young mentioned that in the case of some special restrictions on the parameters (that they do not need to consider), the corresponding integrals should also be approximated using the Airy function. As a result, in order to approximate the function $I(x,\rho)$  it is required to study the integrals of the Airy functions and some exponents. This has been done by using  some standard facts about Airy functions, the saddle point method and tedious calculations.

 Next, we substitute the resulting asymptotic formula for $I(x,\rho)$  in \eqref{eq:doublesum} and perform the change of variables:
\begin{equation}
n-2m=q,\quad n+2m=\rr.
\end{equation}
Since $I(n/m)$ is not negligible only if $q$ is very small compared to $\rr$ (approximately if $q\ll \rr t/(TG)$), we make one more change of variables $\rr=l/q$,  obtaining sums of $\Zag_{l}(1/2+it)$ over the arithmetic progression $l\equiv0\pmod{q}$. To handle this congruence,  we rewrite it with the use of additive harmonics. This allows us  to apply a Voronoi summation formula for the sum over $l$. Arithmetic part of the further computations is identical to that in \cite{BFsym2}. Finally, we obtain  the second moment of Dirichlet $L$-functions, which we estimate using the Heath-Brown large sieve for quadratic characters. The analytic part of computations (after application of Voronoi's formula) is more complex than in \cite{BFsym2}, since the integral transform in Voronoi's formula now includes the Bessel functions $J_{it}(x)$ and $K_{it}(x)$. Therefore, we have to use uniform asymptotic expansions for these Bessel functions, which causes some additional technical difficulties.

The paper is organized as follows. In Section \ref{sec:prelim} we provide  some background information. Section \ref{sec:Voronoi it} is centered around the Voronoi summation formula for  Zagier $L$-series.
Next, in Section \ref{sec:formula for 1moment} we state a formula representing the first moment of symmetric square $L$-functions in terms of sums of Zagier $L$-series. In Section \ref{sec:asympt for 2F1} we collect various asymptotic formulas for the hypergeometric functions appearing in the exact formula for the first moment. In Section \ref{sec:2mom t to LZag} we start proving Theorem \ref{BF t-result} by rewriting the second moment in terms of sums of Zagier $L$-series and by transforming these sums into a form suitable for the application of Voronoi's summation formula.
Section \ref{sec:integral transforms} is devoted to the analysis of various integrals arising as a result of applying the Voronoi summation formula. Finally, in Section \ref{sec:proof of t theorem} we complete the proof of Theorem \ref{BF t-result}.


\section{Notation and preliminaries}\label{sec:prelim}
Let $e(x):=\exp(2\pi ix)$ and let $\Gamma(z)$  denote the Gamma function. For  a fixed $\sigma$ and $|t|\rightarrow\infty$ Stirling's formula yields the asymptotic formula:
\begin{multline}\label{Stirling2}
\Gamma(\sigma+it)=\sqrt{2\pi}|t|^{\sigma-1/2}\exp(-\pi|t|/2)
\exp\left(i\left(t\log|t|-t+\frac{\pi t(\sigma-1/2)}{2|t|}\right)\right)\\\times
\left(1+\sum_{j=1}^{N-1}a_j/t^j+O(|t|^{-N})\right).
\end{multline}

Let $\{u_j\}$ be an orthonormal basis of the space of Maass cusp forms of level one such that each $u_j$ is an eigenfunctions of all Hecke operators and the hyperbolic Laplacian. We denote by $\{\lambda_{j}(n)\}$ the eigenvalues of Hecke operators acting on $u_{j}$ and by $\kappa_{j}=1/4+t_{j}^2$ (with  $t_j>0$)  the eigenvalues of the hyperbolic Laplacian acting on $u_{j}$.
It is known that
\begin{equation*}
u_{j}(x+iy)=\sqrt{y}\sum_{n\neq 0}\rho_{j}(n)K_{it_j}(2\pi|n|y)e(nx),
\end{equation*}
where $K_{\alpha}(x)$ is the $K$-Bessel function and $\rho_{j}(n)=\rho_{j}(1)\lambda_{j}(n).$
The associated symmetric square $L$-function, defined for $\Re{s}>1$ by the series
\begin{equation*}
L(\sym^2 u_{j},s)=\zeta(2s)\sum_{n=1}^{\infty}\frac{\lambda_{j}(n^2)}{n^s},
\end{equation*}
can be  analytically continued to the whole complex plane and has the functional equation
\begin{equation}\label{func.eq}
L_{\infty}(s,t_j)L(\sym^2 u_{j},s)=L_{\infty}(1-s,t_j)L(\sym^2 u_{j},1-s),
\end{equation}
where
\begin{equation}\label{L.infinity}
L_{\infty}(s,t_j)=\pi^{-3s/2}\Gamma\left(\frac{s}{2}\right)\Gamma\left(\frac{s+2it_j}{2}\right)\Gamma\left(\frac{s-2it_j}{2}\right).
\end{equation}
The analytic conductor of $L(\sym^2 u_{j},1/2+2it)$ is equal to
\begin{equation}\label{Lsymsq conductor}
Q:=(1+|t|)(1+|t+t_j|)(1+|t-t_j|).
\end{equation}
Therefore, in the case when $t<(1-\delta)t_j$ the convexity estimate is follows:
\begin{equation}\label{Lsymsq convexity}
|L(\sym^2 u_{j},1/2+2it)|\ll Q^{1/4+\epsilon}\ll (1+|t|)^{1/4+\epsilon}|t_j|^{1/2+\epsilon}.
\end{equation}
Applying \eqref{func.eq} it is possible to derive the approximate functional equation stated in the next lemma.
\begin{lem}
One has
\begin{equation}\label{approx.func.eq.}
L(\sym^2 u_{j},1/2-2it)=\sum_{m=1}^{\infty}\frac{\lambda_j(m^2)}{m^{1/2-2it}}V(n,-t,t_j)+
\frac{L_{\infty}(1/2+2it,t_j)}{L_{\infty}(1/2-2it,t_j)}
\sum_{m=1}^{\infty}\frac{\lambda_j(m^2)}{m^{1/2+2it}}V(n,t,t_j),
\end{equation}
where for any $y>0$ and $a>0$
\begin{equation}\label{approx.fun.eq.Vdef}
V(y,t,t_j)=\frac{1}{2\pi i}\int_{(a)}\frac{L_{\infty}(1/2+2it+z,t_j)}{L_{\infty}(1/2+2it,t_j)}\zeta(1+4it+2z)G(t,z)y^{-z}\frac{dz}{z},
\end{equation}
and
\begin{equation}\label{Gdef}
G(t,z)=\exp(z^2)P_n(t,z^2).
\end{equation}
Here $P_n(t,\cdot)$ is a polynomial of degree n such that $P_n(t,0)=1$, $P_n(t,-4t^2)=0$ and $P_n(t,(1/2\pm2it+2j)^2)=0$ for $j=0,\ldots, n-1$.
\end{lem}
\begin{proof}
The proof is similar to \cite[Lemma 2.2]{Khan} (see also \cite[Lemma 1]{TangXu} and \cite[Lemma 7.2.1]{Ng}).
\end{proof}

Arguing in the same way as in \cite[Lemma 2]{TangXu} and applying  \eqref{L.infinity}, \eqref{Stirling2}, we obtain the following results.
\begin{lem}
For any positive $y,\,t_j$, $0\le|t|<t_j^{1-\epsilon}$ and $A$ we have
\begin{equation}\label{Vestimate}
V(y,t,t_j)\ll\left(\frac{t_j\sqrt{1+|t|}}{y}\right)^{A}.
\end{equation}
For any positive integer $N$ and $1\le y\ll t_j^{1+\epsilon}\sqrt{1+|t|}$, $0\le|t|<t_j^{1-\epsilon}$ the following asymptotic formula holds:
\begin{multline}\label{Vapproximation}
V(y,t,t_j)=
\frac{1}{2\pi i}\int_{(a)}
\left(\frac{\sqrt{t^2_j-t^2}}{\pi^{3/2}y}\right)^z
\frac{\Gamma(1/4+it+z/2)}{\Gamma(1/4+it)}\zeta(1+4it+2z)G(t,z)\\\times
\left(1+\sum_{k=1}^{N-1}\frac{p(t/t_j,v/t_j)}{t_j^k}\right)\frac{dz}{z}+O(t_j^{-N+\epsilon}),
\end{multline}
where $v=\Im(z)$ and $p(x_1,x_2)$ are some rational functions uniformly bounded in $|x_1|,|x_2|\ll1$.
\end{lem}
Note that the part of the integral in \eqref{Vapproximation} over $|\Im(z)|\gg|t_j|^{\epsilon}$ is negligible, so we may assume that $|\Im(z)|\ll|t_j|^{\epsilon}$. Furthermore, we remark that the polynomial $P_n(t,z^2)$ in  \eqref{Gdef} is introduced to cancel the poles of $\Gamma(1/4+it+z/2)\zeta(1+4it+2z)$ in \eqref{Vapproximation}.
\begin{lem}
For $r\sim T,$ $T^{\epsilon}\ll t\ll T^{1-\epsilon}$ one has
\begin{equation}\label{V(t-iN)est}
V(y,t,r-iN)\ll \frac{(1+|t|)^N(1+|r|)^{2N+\epsilon}}{y^{2N}}.
\end{equation}
\end{lem}
\begin{proof}
First, we move the line of  integration in \eqref{approx.fun.eq.Vdef} to $\Re{z}=2N.$ Writing $z=2N+2iv$,  using \eqref{L.infinity} and the estimate $|v|\ll T^{\epsilon/2}$, we infer that
\begin{multline}\label{Linf/Linf r-iN}
\frac{L_{\infty}(1/2+2it+z,t_j)}{L_{\infty}(1/2+2it,t_j)}\ll
\frac{\Gamma\left(\frac{1}{4}+N+i(v+t)\right)\Gamma\left(\frac{1}{4}+2N+i(v+t+r)\right)}
{\Gamma\left(\frac{1}{4}+it\right)\Gamma\left(\frac{1}{4}+N+i(t+r)\right)}
\\\times
\frac{\Gamma\left(\frac{1}{4}+i(v+t-r)\right)}{\Gamma\left(\frac{1}{4}+i(t-r)-N\right)}\ll
t^{N}r^{2N}e^{-\pi v/2}.
\end{multline}
Substituting \eqref{Linf/Linf r-iN} to \eqref{approx.fun.eq.Vdef} and estimating the resulting integral trivially we prove the lemma.
\end{proof}

\begin{lem}\label{lem: Linf/Linf y}
For $t_j=T+Gy$, $|y|\ll \log^2 T$, $G\ll T^{1-\epsilon}/\sqrt{t}$ and $\alt_0=t/T\ll T^{-\delta}$ one  has
\begin{equation}\label{Linfty/Linfty}
\frac{L_{\infty}(1/2+2it,t_j)}{L_{\infty}(1/2-2it,t_j)}=C(t,T)\exp\left(2iGy\log\frac{1+\alt_0}{1-\alt_0}\right)
\sum_{j=0}^{N}\frac{c_j}{T^{j\delta_0}}+O(T^{-A}),
\end{equation}
where $C(t,T),c_j\ll1$, $\delta_0>0$  is some fixed number and $A$ is an arbitrary large constant depending on $N$.
\end{lem}
\begin{proof}
It follows from \eqref{Stirling2} that
\begin{equation}\label{Gamma/Gamma Stirling}
\frac{\Gamma(\sigma+iy)}{\Gamma(\sigma-iy)}=
\exp\left(i\left(2y\log|y|-2t+\frac{\pi y(\sigma-1/2)}{|y|}\right)\right)
\left(1+\sum_{j=1}^{N-1}a_j/y^j+O(|y|^{-N})\right).
\end{equation}
Using \eqref{L.infinity},  \eqref{Gamma/Gamma Stirling} and writing $t=\alt t_j$, we prove that
\begin{equation}\label{Linf/Linf}
\frac{L_{\infty}(1/2+2it,t_j)}{L_{\infty}(1/2-2it,t_j)}=C(t)
\exp\left(2it_j(\left(2\alt\log|t_j|+f(\alt)\right)\right)
\left(1+\sum_{j=1}^{N-1}\frac{a_j}{T^j}+O(T^{-N})\right),
\end{equation}
where
\begin{equation}\label{Linfty/Linfty f def}
f(\alt)=(1+\alt)\log(1+\alt)-(1-\alt)\log(1-\alt)-2\alt=-\alt^3/3+O(\alt^5).
\end{equation}
Note that $$\alt=\frac{t}{t_j}=\frac{t}{T+Gy}=\frac{\alt_0}{1+Gy/T}.$$
Next, we are going to expand the function
$\exp\left(2it_j(\left(2\alt\log t_j+f(\alt)\right)\right)$ in the Taylor series at the point $y=0$.  First, we write the asymptotic formula:
\begin{equation}
4it_j\alt\log t_j=4it\log(T+Gy)=4it\log T+4iGy\alt_0+O\left(\frac{tG^2y^2}{T^2}\right).
\end{equation}
Next, for the sake of simplicity, we let $g(x):=f(\frac{\alt_0}{1+x}).$ Then
\begin{equation}
f(\alt)=g(Gy/T)=g(0)+g'(0)\frac{Gy}{T}+O(\alt^3\frac{G^2y^2}{T^2})=
f(\alt_0)-\alt_0f'(\alt_0)\frac{Gy}{T}+O(\alt^3\frac{G^2y^2}{T^2}).
\end{equation}
Consequently,
\begin{multline}\label{f(alt) in y expansion}
t_jf(\alt)=Tf(\alt_0)+Gy\left(f(\alt_0)-\alt_0f'(\alt_0)\right)+O(\alt_0^3\frac{G^2y^2}{T})=\\=
Tf(\alt_0)+Gy\left(\log\frac{1+\alt_0}{1-\alt_0}-2\alt_0\right)+O(\alt_0^3\frac{G^2y^2}{T}).
\end{multline}
Finally, assuming that $G\ll T^{1-\epsilon}/\sqrt{t}$, we prove that
\begin{multline}
\exp\left(2it_j(\left(2\alt\log t_j+f(\alt)\right)\right)=
\exp\left(4it\log T+2iTf(\alt_0)\right)\\\times
\exp\left(2iGy\log\frac{1+\alt_0}{1-\alt_0}\right)\left(1+O\left(\frac{tG^2y^2}{T^2}\right)\right).
\end{multline}
\end{proof}
\begin{lem}
For $t_j=T+Gy$, $|y|\ll \log^2 T$, $G,t\ll T^{1-\delta_0}$  one  has
\begin{equation}\label{Vapprox y=0}
V(y,t,T+Gy)=V(y,t,T)+\tilde{V}(y,t,t_j)+O(t_j^{-N+\epsilon}),
\end{equation}
\begin{multline}\label{Vtilde}
\tilde{V}(y,t,t_j)=
\frac{1}{2\pi i}\int_{(a)}
\left(\frac{\sqrt{T^2-t^2}}{\pi^{3/2}y}\right)^z\sum_{j=0}^{M}\frac{c_j(z)}{T^{j\delta_0}}
\frac{\Gamma(1/4+it+z/2)}{\Gamma(1/4+it)}\zeta(1+4it+2z)G(t,z)\\\times
\left(1+\sum_{k=1}^{N-1}\frac{p(t/t_j,v/t_j)}{t_j^k}\right)\frac{dz}{z}+O(t_j^{-N+\epsilon}),
\end{multline}
where $c_j(z)\ll1$.
\end{lem}
\begin{proof}
First, we truncate the integral in \eqref{Vapproximation} to the region $|\Im{z}|\ll T^{\epsilon}$ producing a negligibly small error term. Now the statement of the lemma follows from the formula
\begin{equation}
\left(t_j^2-t^2\right)^z=\left(T^2-t^2+2TGy+G^2y^2\right)^z=\left(T^2-t^2\right)^z\left(1+\frac{2TGy+G^2y^2}{T^2-t^2}\right)^z
\end{equation}
after expanding the last bracket in the Taylor series.
\end{proof}
According to \cite[Eq. 3.323.2,Eq. 3.462.2]{GR} we have
\begin{equation}\label{exp integral1}
\int_{-\infty}^{\infty}\exp(-p^2y^2+qy)dy=\frac{\pi^{1/2}}{p}\exp\left(\frac{q^2}{4p^2}\right),
\quad\hbox{if}\quad\Re(p^2)>0,
\end{equation}
\begin{equation}\label{exp integral2}
\int_{-\infty}^{\infty}x^n\exp(-y^2+qy)dy=P_n(q)\exp\left(\frac{q^2}{4}\right),
\end{equation}
where $n$ is a nonnegative integer and $P_n(q)$ is a polynomial of degree $n.$

For $\Re{s}>1$ we define the Zagier $L$-series by
\begin{equation}\label{Lbyk}
\Zag_{n}(s)=\frac{\zeta(2s)}{\zeta(s)}\sum_{q=1}^{\infty}\frac{\rho_q(n)}{q^{s}},\quad
\rho_q(n):=\#\{x\Mod{2q}:x^2\equiv n\Mod{4q}\}.
\end{equation}

Zagier \cite[Proposition 3]{Z} proved a  meromorphic continuation of  \eqref{Lbyk} to the whole complex plane and a functional equation for $\Zag_{n}(s)$ relating the values $s$ and $1-s$.

Another useful property is that for $n=Dl^2$ with $D$ fundamental discriminant one has
\begin{equation}\label{ldecomp}
\Zag_{n}(s)=l^{1/2-s}T_{l}^{(D)}(s)L(s,\chi_D),
\end{equation}
where $L(s,\chi_D)$ is a Dirichlet $L$-function for primitive quadratic character $\chi_D$ and
\begin{equation}\label{eq:td}
T_{l}^{(D)}(s)=\sum_{l_1l_2=l}\chi_D(l_1)\frac{\mu(l_1)}{\sqrt{l_1}}\tau_{s-1/2}(l_2).
\end{equation}

To estimate various integrals  we will use the following lemma (see \cite[Lemma 8.1]{BKY} and \cite[Lemma A1]{AHLQ}).
\begin{lem}\label{Lemma BKY}
Suppose that there are parameters $P,R,V,X,Y>0$ such that
\begin{equation}\label{BKYconditions}
|f'(x)|\gg R,\quad
f^{(i)}(x)\ll\frac{Y}{P^i},\quad
g^{(j)}(x)\ll\frac{X}{V^j}
\end{equation}
for $i\ge2,j\ge0$. Then
\begin{equation}\label{I BKY est}
I=\int_a^{b}g(x)e(f(x))dx\ll(b-a)X\left(\frac{1}{RV}+\frac{1}{RP}+\frac{Y}{R^2P^2}\right)^{A}.
\end{equation}
\end{lem}

For the sake of simplicity, let us consider the following conventions. Quite often in this paper we study functions of several variables and parameters and obtain uniform asymptotic expansions  with respect to one of the parameters, say
\begin{equation}
f(x_1,\ldots,x_k, T)=\sum_{j=0}^{n}\frac{c(x_1,\ldots,x_k, T)}{T^j}+O(T^{-n-1}).
\end{equation}
In this case, we usually only consider the contribution of the main term $c(x_1,\ldots,x_k, T)$. Further, if we are not interested in dependence on some variables, we use the following notation:
\begin{equation}
f(x_1,\ldots,x_k, T)\sim c(x_{i_1},\ldots,x_{i_l}).
\end{equation}

\section{Voronoi's summation formula for $\Zag_n(1/2+2it)$}\label{sec:Voronoi it}
The arithmetic structure of  Voronoi's formula for  $\Zag_n(1/2+2it)$ is the same as for the central values $\Zag_n(1/2)$. As explained in \cite{BFsym2}, it is sufficient to consider
only the case of $c\equiv0\pmod{4}$  as all other cases can be treated similarly. Therefore, we provide detailed computations only for $c\equiv0\pmod{4}$ in order to demonstrate required changes in the analysis of integral transforms.

Following Dunster \cite[(3.4a), (3.4b)]{Dun90Bessel}, we define
\begin{equation}\label{FG Dun def}
F_{\mu}(z):=\frac{J_{\mu}(z)+J_{-\mu}(z)}{2\cos(\pi\mu/2)},\quad
G_{\mu}(z):=\frac{J_{\mu}(z)-J_{-\mu}(z)}{2\sin(\pi\mu/2)}.
\end{equation}
Let $L_{\mu}(z)$ be either $F_{\mu}(z)$ or $G_{\mu}(z)$. The function $\omega_L(z)=L_{2it}(2tz)\sqrt{z}$  satisfies the differential equation (see \cite[(5.1)]{Dun90Bessel}):
\begin{equation}\label{FG Dun difeq}
\frac{d^2\omega_L}{dz^2}=-\frac{1+16t^2(1+z^2)}{4z^2}\omega_L(z).
\end{equation}
The following uniform asymptotic formulas take place (see  \cite[(5.15),(5.16)]{Dun90Bessel}):
\begin{equation}\label{GF2tz asympt}
L_{2it}(z)=\frac{1}{\sqrt{\pi t}(1+z^2)^{1/4}}\sum_{\pm}e^{\pm2it\xi(z)}E_{\pm}(t,(1+z^2)^{-1/2}),
\end{equation}
where
\begin{equation}\label{GF2tz xi def}
\xi(v)=\sqrt{1+v^2}+\log\frac{v}{1+\sqrt{1+v^2}},\quad
\xi'(v)=\frac{\sqrt{1+v^2}}{v},
\end{equation}
and the non-oscillatory functions $E_{\pm}(t,p)$ have the Poincare type asymptotic expansion (one can express these functions into a series with decreasing powers of $t$ and with an error that is smaller than the last term of the expansion).

Similar results hold for the $K$-Bessel function. The function $\omega_K(z)=K_{2it}(2tz)\sqrt{z}$  satisfies the differential equation (see \cite[(4.1)]{Dun90Bessel}):
\begin{equation}\label{K Dun difeq}
\frac{d^2\omega_K}{dz^2}=\frac{1+16t^2(z^2-1)}{4z^2}\omega_K(z).
\end{equation}
For $1-z\gg t^{-2/3+\epsilon}$ the following uniform asymptotic formula takes place (see \cite{Balogh}, \cite[sec.4]{Dun90Bessel}):
\begin{multline}\label{K2tz asympt}
e^{\pi t}K_{2it}(2tz)=\frac{\sqrt{\pi}}{\sqrt{t}(1-z^2)^{1/4}}
\Bigl(
\cos(2t\eta(z)-\frac{\pi}{4})\sum_{s=0}^n\frac{U_{2s}((1-z^2)^{-1/2})}{(2t)^{2s}}+\\+
\sin(2t\eta(z)-\frac{\pi}{4})\sum_{s=0}^{n-1}\frac{U_{2s+1}((1-z^2)^{-1/2})}{(2t)^{2s+1}}
+O\left(\frac{1}{(t(1-z^2)^{1/2})^{2n+1}}\right)
\Bigr),
\end{multline}
where $U_j(x)=O(x^j)$ and
\begin{equation}\label{K2tz eta def}
\eta(v)=-\sqrt{1-v^2}+\log\frac{1+\sqrt{1-v^2}}{v},\quad
\eta'(v)=-\frac{\sqrt{1-v^2}}{v}.
\end{equation}
In the remaining ranges of $z$ we apply the results of \cite{BST}. More precisely, for $z-1\gg t^{-2/3+\epsilon}$
\begin{equation}\label{K2tz z>1 est}
e^{\pi t}K_{2it}(2tz)\ll (zt)^{-A}.
\end{equation}
For $|x-2t|\ll t^{1/3+\epsilon}$ one has
\begin{equation}\label{K2tz z=1 est}
e^{\pi t}K_{2it}(x)\ll \frac{1}{t^{1/3}},\quad
e^{\pi t}\frac{\partial^j}{\partial x^j}K_{2it}(x)\ll\frac{(t^{2/3}/x)^j}{t^{1/3}}.
\end{equation}

In the above settings, for $k=1/2$ and $\rho=it$ the functions $\Phi^{(\pm,\pm)}(z)$ ( see \cite[(1.9)]{BFVoron}) can be written in the following form :
\begin{equation}\label{Phi++ itdef}
\Phi^{(+,+)}(z)=\frac{\sqrt{z}}{\sqrt{2}}\left(F_{2it}(2\sqrt{z})-G_{2it}(2\sqrt{z})
\right),
\end{equation}
\begin{equation}\label{Phi-- itdef}
\Phi^{(-,-)}(z)=\frac{-\sqrt{z}}{\sqrt{2}}\left(F_{2it}(2\sqrt{z})+G_{2it}(2\sqrt{z})
\right).
\end{equation}

We also modify slightly the definition of $\Phi^{(\pm,\mp)}(z)$ (see \cite[(1.10)]{BFVoron}) putting
\begin{equation}\label{Phi-+ itdef}
\Phi^{(-,+)}(z):=\frac{2\sqrt{z}K_{2it}(2\sqrt{z})}{\Gamma(1/4+it)\Gamma(3/4-it)},
\end{equation}
\begin{equation}\label{Phi+- itdef}
\Phi^{(+,-)}(z):=\frac{2\sqrt{z}K_{2it}(2\sqrt{z})}{\Gamma(1/4-it)\Gamma(3/4+it)}.
\end{equation}
To formulate the Voronoi formula it is also required to define the theta multiplier (see \cite[p.2380]{Strom}):
\begin{equation}\label{ThetaMultiplier def}
 \TM(M):=\bar{\epsilon}_{d}\left(\frac{c}{d}\right),
\end{equation}
where the symbol $\left(\frac{c}{d}\right)$ is given on \cite[p.442]{Shimura} (for positive odd  $d$  it coincides with a classical Jacobi  symbol) and $\epsilon_{d}=1$ if $d\equiv1\pmod 4$ and $\epsilon_{d}=i$ if $d\equiv-1\pmod 4$.

\begin{lem}\label{Thm Voronoi+ it c0mod4}
For $c\equiv0\pmod{4}$ and $ad\equiv1\pmod{4}$ one has
\begin{multline}\label{Voronoi it n>0 c=0 mod 4}
\sum_{n>0}\phi(n)e\left( \frac{an}{c}\right)\frac{\Zag_n(1/2+2it)}{n^{1/2-it}}=
\TM(M_0)e(1/8)\sqrt{2}\Res_{+}(c,t)+\\+
\TM(M_0)e(1/8)\sum_{n\neq0}\widehat{\phi}^{+}\left(\frac{4\pi^2n}{c^2}\right)e\left( -\frac{dn}{c}\right)
\frac{\Zag_n(1/2+2it)}{|n|^{1/2-it}},
\end{multline}
\begin{multline}\label{Voronoi it n<0 c=0 mod 4}
\sum_{n<0}\phi(n)e\left( \frac{an}{c}\right)\frac{\Zag_n(1/2+2it)}{|n|^{1/2-it}}=
\TM(M_0)e(1/8)\sqrt{2}\Res_{-}(c,t)+\\+
\TM(M_0)e(1/8)\sum_{n\neq0}\widehat{\phi}^{-}\left(\frac{4\pi^2n}{c^2}\right)e\left( -\frac{dn}{c}\right)
\frac{\Zag_n(1/2+2it)}{|n|^{1/2-it}},
\end{multline}
where
\begin{equation}\label{Res+ c=0(4)}
\Res_{\pm}(c,t)=
\frac{2^{4it}\phi^{\pm}(1/2+it)}{c^{1+2it}}\zeta(1+4it)+
\frac{\pi^{2it}\phi^{\pm}(1/2-it)}{c^{1-2it}}\zeta(1-4it)\frac{\Gamma\left(\frac{2\pm1}{4}+it\right)\Gamma(1/2-2it)}{\Gamma\left(\frac{2\pm1}{4}-it\right)\Gamma(1/2+2it)},
\end{equation}
\begin{equation}\label{phi+ transform def}
\phi^{+}(s)=\int_{0}^{\infty}\phi(\pm x)x^{s-1}dx,\quad
\widehat{\phi}^{+}(\pm y)=\int_{0}^{\infty}\frac{\phi(x)}{x}\Phi^{(\pm,+)}(xy)dx,\quad
\end{equation}
\begin{equation}\label{phi- transform def}
\phi^{-}(s)=\int_{0}^{\infty}\phi(-x)x^{s-1}dx,\quad
\widehat{\phi}^{-}(\pm y)=\int_{0}^{\infty}\frac{\phi(-x)}{x}\Phi^{(\pm,-)}(xy)dx.
\end{equation}
\end{lem}
\begin{proof}
This follows  from \cite[Theorem 1.2, Corollary 1.5]{BFVoron} after some straightforward calculations.
\end{proof}

For $c\equiv0\pmod{4}$ and $ad\equiv1\pmod{4}$,   let
\begin{equation}\label{STM 0 def}
\STM_0(c,n):=\sum_{\substack{a\pmod{c}\\(a,c)=1}}\TM(M_0)e\left( -\frac{dn}{c}\right)=
\sum_{\substack{a\pmod{c}\\(a,c)=1}}\bar{\epsilon}_{d}\left(\frac{c}{d}\right)e\left( -\frac{dn}{c}\right).
\end{equation}
According to \cite[(3.10)]{BFsym2}
\begin{equation}\label{STM0 series2}
\sum_{c\equiv0\pmod{4}}\frac{\STM_0(c,n)}{c^{2s}}=\frac{L^{\ast}(2s-1/2,n)}{\zeta_2(4s-1)}\overline{r_2(\bar{s},n)},
\end{equation}
where the function $r_2(s,n)$ is defined in \cite[(3.6),(3.7)]{PRR}, the series $L^{\ast}(z)$ is defined in \cite[(2.8)]{PRR} and it coincides up to a finite product with a Dirichlet $L$-function. As shown in \cite[Sec.3]{BFsym2}, $L^{\ast}(z)$ and $\Zag_n(z)$ differ  slightly from each other.

Using Heath-Brown's large sieve inequality \cite{HB} (see \cite[(3.1)]{KhYoung} or \cite[(2.28)]{PRR}) one obtains ( see \cite[(3.2)]{KhYoung}, \cite[Theorem 2.9, Remark 2.10]{PRR}) the following upper bounds:
\begin{equation}\label{Last 2mom estimate}
\sum_{n\le N}|L^{\ast}(1/2+it,n)|^2\ll\left(N+\sqrt{N(1+|t|)}\right)\left(N(1+|t|)\right)^{\epsilon},
\end{equation}
\begin{equation}\label{LZag 2mom estimate}
\sum_{n\le N}|\Zag_n(1/2+it)|^2\ll\left(N+\sqrt{N(1+|t|)}\right)\left(N(1+|t|)\right)^{\epsilon}.
\end{equation}


\section{Reciprocity formula for the first twisted moment}\label{sec:formula for 1moment}
Let
\begin{equation}
\M_1(l,\rho;h):=\sum_{j}h(t_j)\alpha_{j}\lambda_{j}(l^2)L(\sym^2 u_{j},\rho),
\end{equation}
where  $h(t)$ is a weight function satisfying the following conditions:
\begin{description}
  \item[C1] $h(r)$ is an even function;
  \item[C2] $h(r)$ is holomorphic in the strip $|\Im(r)|<\Upsilon$ for some $\Upsilon>1/2$;
  \item[C3] $h(r)$ is such that the following estimate $h(r)\ll(1+|r|)^{-2-\epsilon}$ holds in the strip $|\Im(r)|<\Upsilon$, $\Upsilon>1/2$;
  \item[C4] $h(\pm(n+1/2)i)=0$ for $n=0,1,\ldots N-1$, where $N>0$ is a sufficiently large integer.
\end{description}
A standard choice for the function $h(r)$ is
\begin{equation}\label{hN def}
h(T,G,N;r):=q_N(r)\exp\left(-\frac{(r-T)^2}{G^2}\right)+q_N(r)\exp\left(-\frac{(r+T)^2}{G^2}\right),
\end{equation}
where for an arbitrary large integer $N$
\begin{equation}\label{qN def}
q_N(r):=\frac{(r^2+1/4)\ldots(r^2+(N-1/2)^2)}{(r^2+100N^2)^N}.
\end{equation}
Note that $q_N(r)=1+O(1/r^2)$. To simplify the notation, we write $h(r)$ instead of $h(T,G,N;r)$.
For  $x \geq 2$, let
\begin{multline}\label{integralIgeq2}
I(x,\rho;h):=\frac{2^{2-\rho}i}{\pi^{3/2}}\int_{-\infty}^{\infty}\frac{rh(r)}{\cosh(\pi r)}\left(\frac{2}{x}\right)^{2ir}
\frac{\Gamma(1/2-\rho/2+ir)\Gamma(1-\rho/2+ir)}{\Gamma(1+2ir)}\\ \times \sin\left( \pi(\rho/2-ir)\right) {}_2F_{1}\left(1/2-\rho/2+ir,1-\rho/2+ir,1+2ir;\frac{4}{x^2} \right)dr,
\end{multline}
\begin{multline}\label{integralIeq2}
I(2,\rho;h):=\Gamma(\rho-1/2)
\frac{2^{2-\rho}i}{\pi^{3/2}}\int_{-\infty}^{\infty}\frac{rh(r)}{\cosh(\pi r)}
\sin\left( \pi(\rho/2-ir)\right)\\ \times \frac{\Gamma(1/2-\rho/2+ir)\Gamma(1-\rho/2+ir)}{\Gamma(\rho/2+ir)\Gamma(1/2+\rho/2+ir)}dr
\end{multline}
and for $0<x<2$
\begin{multline}\label{integralIleq2}
I(x,\rho;h):=\frac{2i}{\pi^{3/2}}\int_{-\infty}^{\infty}\frac{rh(r)}{\cosh(\pi r)}x^{1-\rho}
\frac{\Gamma(1/2-\rho/2+ir)\Gamma(1/2-\rho/2-ir)}{\Gamma(1/2)}\\ \times \cos\left( \pi(\rho/2+ir)\right) {}_2F_{1}\left(1/2-\rho/2+ir,1/2-\rho/2-ir,1/2;\frac{x^2}{4} \right)dr.
\end{multline}
According to \cite[Theorem 5]{Bal} we following explicit formula holds.
\begin{thm}\label{thm rho=1/2 exact} For  any function $h(t)$  satisfying the conditions $(C1)-(C4)$ and $\Re{\rho}=1/2$ we have
\begin{equation}\label{eq:M1lrho=1/2}
M_1(m,\rho;h)=MT(m,\rho;h)+CT(m,\rho;h)+ET(m,\rho;h)+S_1(m,\rho;h)+S_2(m,\rho;h),
\end{equation}
where
\begin{equation}\label{M1 MT}
MT(m,\rho;h)=\frac{\zeta(2\rho)}{\pi^2m^{\rho}}\int_{-\infty}^{\infty}rh(r)\tanh(\pi r)dr+
(2\pi)^{\rho-1}\zeta(2\rho-1)\frac{I(2,\rho;h)}{(2m)^{1-\rho}},
\end{equation}
\begin{equation}\label{M1 CT}
CT(m,\rho;h)=-\frac{\zeta(\rho)}{\pi}\int_{-\infty}^{\infty}\frac{\tau_{ir}(m^2)\zeta(\rho+2ir)\zeta(\rho-2ir)}{|\zeta(1+2ir)|^2}h(r)dr,
\end{equation}
\begin{multline}\label{M1 ET}
ET(m,\rho;h)=-\frac{2\zeta(2\rho-1)}{\zeta(2-\rho)}\tau_{(1-\rho)/2}(m^2)h\left(\frac{1-\rho}{2i}\right)\\
+\Zag_{-4m^2}(\rho)\frac{2^{1-\rho}i}{(4\pi m)^{1-\rho}\pi}
\int_{-\infty}^{\infty}\frac{rh(r)}{\cosh(\pi r)}\frac{\Gamma(1/2-\rho/2+ir)}{\Gamma(1/2+\rho/2+ir)}dr,
\end{multline}
\begin{equation}\label{M1 S1}
S_1(m,\rho;h)=(2\pi)^{\rho-1}\sum_{n=1}^{2m-1}\frac{1}{n^{1-\rho}}\mathscr{L}_{n^2-4m^2}(\rho)I\left(\frac{n}{m},\rho;h\right),\quad
\end{equation}
\begin{equation}\label{M1 S2}
S_2(m,\rho;h)=(2\pi)^{\rho-1}\sum_{n=2m+1}^{\infty}\frac{1}{n^{1-\rho}}\mathscr{L}_{n^2-4m^2}(\rho)I\left(\frac{n}{m},\rho;h\right).
\end{equation}
\end{thm}




\section{Asymptotic formulas for hypergeometric functions}\label{sec:asympt for 2F1}
For the hypergeometric function appearing in \eqref{integralIgeq2} the following asymptotic formula holds (see \cite{BF2F1 I}).
\begin{prop}\label{prop: 2F1 asympt1}
For $0<z<1,$ $r\to\infty$ and $r^{-1+\delta}\ll\alt\ll r^{-\delta}$ one has
\begin{multline}\label{2F1 main asymptI}
\HyG\left(1/4+ir(1-\alt),3/4+ir(1-\alt),1+2ir;z\right)=\\=
\frac{\exp\left(2ir\Lo(\alt,z)\right)}{(1-(1-\alt^2)z)^{1/4}}
\left(1+\sum_{j=1}^{N}\frac{c_j(\alt,r)}{(\alt r)^j}\right)+O((\alt r)^{-N-1}),
\end{multline}
where
\begin{equation}\label{2F1 main asymptI Lo func}
\Lo(\alt,z)=\log2-\alt\log(1+\alt)-\log(1+\sqrt{1-(1-\alt^2)z})+\alt\log(\alt+\sqrt{1-(1-\alt^2)z}).
\end{equation}
\end{prop}
To simply the notation, we use the sign $\sim$ to indicate the main term of an asymptotic expansion up to some constants.
\begin{cor}\label{cor: 2F1 asympt1}
For $0<z<1,$  $r=T+Gy$, $|y|\ll \log^2 T$, $G\ll T^{1/2-\epsilon}$ and $\alt_0=t/T\ll T^{-\delta}$, the main term in the asymptotic expansion \eqref{2F1 main asymptI} can be rewritten as
\begin{multline}\label{2F1 main asymptIy}
\HyG\left(1/4+ir(1-\alt),3/4+ir(1-\alt),1+2ir;z\right)\sim
(1-(1-\alt_0^2)z)^{-1/4}\\\times
e^{2iT\left(-\log(1+\sqrt{1-(1-\alt_0^2)z})+\alt_0\log(\alt_0+\sqrt{1-(1-\alt_0^2)z})\right)}
e^{2iGy\left(-\log(1+\sqrt{1-(1-\alt_0^2)z})+\log2\right)}.
\end{multline}

\end{cor}
\begin{proof}
The statement can be proved similarly to Lemma \ref{lem: Linf/Linf y}. As in  \eqref{f(alt) in y expansion} one has
\begin{equation}\label{rlo in y}
r\Lo(\alt,z)=T\Lo(\alt_0,z)
+Gy\left(\Lo(\alt_0,z)-\alt_0\Lo'(\alt_0,z)\right)+O\left(\frac{G^2y^2}{T}\right),
\end{equation}
where the derivative is taken over $\alt.$ Direct calculations (see \eqref{2F1 main asymptI Lo func}) show that
\begin{equation}\label{lo-alt lo'}
\Lo(\alt_0,z)-\alt_0\Lo'(\alt_0,z)=\log2-\log(1+\sqrt{1-(1-\alt_0^2)z}).
\end{equation}
Substituting \eqref{2F1 main asymptI Lo func} and \eqref{lo-alt lo'}  to \eqref{rlo in y} we obtain \eqref{2F1 main asymptIy}.
\end{proof}

Let
\begin{multline}\label{2f1 1/4 1/4 1/2 def}
\Fs(r,\alt,x):=
\frac{\Gamma(1/4-it+ir)\Gamma(1/4-it+ir)}{\Gamma(1/2)}\HyG\left(1/4-it+ir,1/4-it-ir,1/2;x \right)=\\=
\frac{\Gamma(1/4+ir(1-\alt))\Gamma(1/4-ir(1+\alt))}{\Gamma(1/2)}\HyG\left(1/4+ir(1-\alt),1/4-ir(1+\alt),1/2;x \right).
\end{multline}
This function appears in \eqref{integralIleq2} and its asymptotic behavior was studied in \cite{BF2F1 II}. In particular, \cite[Corollary 1.3]{BF2F1 II} states that the leading term of the asymptotic expansion has the following form.

\begin{prop}\label{prop: y=1-alt^2}
For $0<y<1$, $r\to\infty$ and $r^{1-\delta}\ll\alt\ll r^{-\delta}$ one has
\begin{equation}\label{2f1 1/4 1/4 1/2 asympt y=1-alt^2}
\Fs(r,\alt,y)\sim
e^{ir\Lt(\alt,y)}
\frac{ 2^{3/2}\pi e^{-\pi r\alt}}{(2r\alt)^{1/3}}\left(\frac{\alt^2\hat{\zeta}(y)}{y-1+\alt^2}\right)^{1/4}Ai(-(2r\alt)^{2/3}\hat{\zeta}(y)),
\end{equation}
where
\begin{equation}\label{zeta(y) def}
\hat{\zeta}(y)=\left\{
            \begin{array}{ll}
              -\left(3\Aa_0(\alt,y)/(2\alt)\right)^{2/3}, & \hbox{if} \quad y<1-\alt^2 \\
              \left(3\Aa_1(\alt,y)/(2\alt)\right)^{2/3}, & \hbox{if} \quad y>1-\alt^2,
            \end{array}
          \right.
\end{equation}
\begin{equation}\label{Aa0 def}
\Aa_0(\alt,y)=
\frac{\pi(1-\alt)}{2}-\arctan\frac{\sqrt{y}}{\sqrt{1-\alt^2-y}}+\alt\arctan\frac{\alt\sqrt{y}}{\sqrt{1-\alt^2-y}},
\end{equation}
\begin{multline}\label{Aa1 def}
\Aa_1(\alt,y)=
\alt\log\left(\alt\sqrt{y}+\sqrt{y-1+\alt^2}\right)-\log\left(\sqrt{y}+\sqrt{y-1+\alt^2}\right)-\\-\frac{\alt}{2}\log(1-y)+
\frac{1-\alt}{2}\log(1-\alt^2),
\end{multline}
\begin{equation}\label{Aa0 Taylor}
\Aa_0(\alt,y)=
\frac{(1-\alt^2)(1-\alt^2-y)^{3/2}}{3\alt^2y^{3/2}}\left(1+O\left(\frac{1-\alt^2-y}{\alt^2}\right)\right).
\end{equation}
\end{prop}
Note that by \cite[p.10]{BF2F1 II} one has
\begin{equation}\label{Aa1 Taylor}
\Aa_1(\alt,y)=
\frac{(y-1+\alt^2)^{3/2}}{3\alt^2\sqrt{1-\alt^2}}\left(1+O\left(\frac{y-1+\alt^2}{\alt^2}\right)\right).
\end{equation}

In the case when the argument of the Airy function is bigger than $r^{\delta}$ it is possible to apply asymptotic formulas \cite[(9.7.5),(9.7.9)]{HMF} for the Airy function as shown in
\cite[Corollary 1.2, 1.3]{BF2F1 II}.
\begin{cor}\label{cor: y<1-alt^2}
For $0<y<1-\alt^2-\frac{\alt^{4/3}}{r^{2/3-\delta}}$ one has
\begin{equation}\label{2f1 1/4 1/4 1/2 asympt y<1-alt^2}
\Fs(r,\alt,y)\ll
\frac{e^{-\pi r\alt}e^{-2r\Aa_0(\alt,y)}}{\sqrt{r}(1-\alt^2-y)^{1/4}}\ll\frac{e^{-\pi r\alt}}{r^{A}}.
\end{equation}
\end{cor}
\begin{cor}\label{cor: y>1-alt^2}
For $1-\alt^2+\frac{\alt^{4/3}}{r^{2/3-\delta}}<y<1$
the leading coefficient in the asymptotic expansion for $\Fs(r,\alt,y)$ is given by
\begin{equation}\label{2f1 1/4 1/4 1/2 asympt y>1-alt^2}
\Fs(r,\alt,y)\sim
2\sqrt{\pi}e^{ir\Lt(\alt,y)}e^{-\pi r\alt}
\frac{\cos\left(2r\Aa_1(\alt,y)-\pi/4\right)}{\sqrt{r}(y-1+\alt^2)^{1/4}},
\end{equation}
where
\begin{equation}\label{2F1 main asymptII Lt func}
\Lt(\alt,y)=\alt\log(1-y)-2\alt\log r+(1-\alt)\log(1-\alt)-(1+\alt)\log(1+\alt)+2\alt.
\end{equation}
\end{cor}

\section{The second moment on the critical line and sums of Zagier $L$-series}\label{sec:2mom t to LZag}
In order to estimate
\begin{equation}
\sum_{T<t_j\le T+G}\alpha_{j}L(\sym^2 u_{j},1/2-2it)L(\sym^2 u_{j},1/2+2it),
\end{equation}
we represent the first $L$-function via approximate functional equation \eqref{approx.func.eq.}, and then apply Theorem \ref{eq:M1lrho=1/2} with $\rho=1/2+2it$ and a test function $h(r)$ being either
\begin{equation}\label{H0 Hinf def}
H_0(t,r)=V(m,-t,r)h(T,G,N;r)\quad\hbox{or}\quad
H_{\infty}(t,r)=\frac{L_{\infty}(1/2+2it,r)}{L_{\infty}(1/2-2it,r)}V(m,t,r)h(T,G,N;r),
\end{equation}
where $h(T,G,N;r)$ is given by \eqref{hN def} and \eqref{qN def}.  Since the function $H_{\infty}(t,r)$ does not satisfy the conditions in the beginning of Section \ref{sec:formula for 1moment}, we first replace it by \eqref{Linf/Linf}. It is enough to consider only the main term of this approximation, since all other terms can be treated in the same way. Therefore,
\begin{equation}\label{Hinf def}
H_{\infty}(t,r)=
e^{2ir(\left(2\alt\log|r|+f(\alt)\right)}
V(m,t,r)h(T,G,N;r),
\end{equation}
where $f(\alt)$ is given by \eqref{Linfty/Linfty f def} and $\alt=t/r.$ All summands in \eqref{eq:M1lrho=1/2} except
$S_1(m,\rho;h)+S_2(m,\rho;h),$ can be easily bounded by $T^{1+\epsilon}G$. The contribution of $CT(m,\rho;h)$ can be estimated as in \cite[Lemma 5.3]{BF2mom}. In fact, arguing as in \cite[Lemma 5.3]{BF2mom} we obtain a product of five zeta-functions (one extra zeta-function appears because of the presence of  $\zeta(\rho)$ in \eqref{M1 CT}). Using the bound $|\zeta(1/2+iy)|\ll y^{1/6}$, and further estimating the integral over the short interval trivially, we obtain $T^{5/6+\epsilon}G.$ Consequently,
\begin{equation}\label{t2mom est S1+S2}
\sum_{T<t_j\le T+G}\alpha_{j}|L(\sym^2 u_{j},1/2+2it)|^2\ll T^{1+\epsilon}G+S_1(t)+S_2(t),
\end{equation}
where
\begin{equation}\label{S1 def}
S_1(t)=
\sum_{m\ll T^{1+\epsilon}\sqrt{t}}\sum_{n=1}^{2m-1}\frac{\mathscr{L}_{n^2-4m^2}(1/2+2it)}{n^{1/2-2it}}
\left(\frac{I\left( \frac{n}{m},1/2+2it;H_0\right)}{m^{1/2-2it}}+\frac{I\left( \frac{n}{m},1/2+2it;H_{\infty}\right)}{m^{1/2+2it}}\right),
\end{equation}
\begin{equation}\label{S2 def}
S_2(t)=
\sum_{m\ll T^{1+\epsilon}\sqrt{t}}\sum_{n=2m+1}^{\infty}\frac{\mathscr{L}_{n^2-4m^2}(1/2+2it)}{n^{1/2-2it}}
\left(\frac{I\left( \frac{n}{m},1/2+2it;H_0\right)}{m^{1/2-2it}}+\frac{I\left( \frac{n}{m},1/2+2it;H_{\infty}\right)}{m^{1/2+2it}}\right).
\end{equation}


\subsection{Preliminary estimate of $S_1(t)$}\label{subsec:S1}
Let us use the following notation:
\begin{equation}
\rho:=1/2+2it,\quad\alt:=\frac{t}{r},\quad y:=\frac{n^2}{4m^2}.
\end{equation}
Using definitions \eqref{integralIleq2} and \eqref{2f1 1/4 1/4 1/2 def} we infer that
\begin{equation}\label{Ileq2 as1}
I\left(\frac{n}{m},\rho;h\right)\sim\int_{-\infty}^{\infty}\frac{rh(r)}{\cosh(\pi r)}\cos\left( \pi(\rho/2+ir)\right)y^{1/2-\rho/2}
\Fs(r,\alt,y)dr.
\end{equation}
For simplicity, we only consider the contribution of the first exponent in \eqref{hN def}. Therefore, the contribution of $r<0$ to \eqref{Ileq2 as1} is negligible and we can assume that $r>0$
\begin{equation}\label{Ileq2 as2}
I\left(\frac{n}{m},\rho;h\right)\sim\int_{-\infty}^{\infty}rh(r)y^{1/4-ir\alt}e^{\pi r\alt}\Fs(r,\alt,y)dr.
\end{equation}
We make the change of variables $r=T+Gv$ in \eqref{Ileq2 as2}. It immediately follows from \eqref{hN def} that the contribution of $|v|\gg\log^2T$ is negligible, so from now on we assume that
\begin{equation}
r=T+Gv,\quad |v|\ll\log^2T,\quad
\alt=\alt(x)=\frac{\alt_0}{1+x},\quad x=Gv/T,\quad \alt_0=t/T.
\end{equation}
Our next step is to apply Corollaries \ref{cor: y<1-alt^2} and \ref{cor: y>1-alt^2}. A small difficulty arises while verifying the conditions of these corollaries because now $\alt$  depends on $v$. Nevertheless, for $G\ll T^{1+\delta}/t^{2/3}$ one has
\begin{equation}
\frac{\alt_0^{4/3}}{T^{2/3-2\delta}}\gg\alt_0^2\frac{Gv}{T},
\end{equation}
and therefore,
\begin{equation}\label{y>1-alt0^2+  condition}
y>1-\alt_0^2+\frac{\alt_0^{4/3}}{T^{2/3-2\delta}}>1-\alt^2+\frac{\alt^{4/3}}{T^{2/3-\delta}},
\end{equation}
\begin{equation}\label{y<1-alt0^2+  condition}
y<1-\alt_0^2-\frac{\alt_0^{4/3}}{T^{2/3-2\delta}}<1-\alt^2-\frac{\alt^{4/3}}{T^{2/3-\delta}}.
\end{equation}
In the case of \eqref{y<1-alt0^2+  condition}, it follows from \eqref{Ileq2 as2} and \eqref{2f1 1/4 1/4 1/2 asympt y<1-alt^2}
that $$I\left(\frac{n}{m},\rho;h\right)\ll T^{-A}.$$

Substituting \eqref{2f1 1/4 1/4 1/2 asympt y=1-alt^2} to \eqref{Ileq2 as2} and using \eqref{2F1 main asymptII Lt func} we show that
\begin{equation}\label{Ileq2 as3}
I\left(\frac{n}{m},\rho;h\right)\sim
y^{1/4-it}(1-y)^{it}G
\int_{-\infty}^{\infty}rh(r)
\frac{e^{-2it\log r+ir\Lt(\alt)}}{(r\alt)^{1/3}}\left(\frac{\alt^2\hat{\zeta}(y)}{y-1+\alt^2}\right)^{1/4}Ai(-(2r\alt)^{2/3}\hat{\zeta}(y))
dv,
\end{equation}
where we write $r$ instead of $T+Gv$ to shorten the notation and
\begin{equation}\label{Lt alt def}
\Lt(\alt)=(1-\alt)\log(1-\alt)-(1+\alt)\log(1+\alt)+2\alt.
\end{equation}
Applying \eqref{Vapprox y=0} and \eqref{Linfty/Linfty} one has
\begin{multline}\label{Ileq2H0 as4}
I\left(\frac{n}{m},\rho;H_0\right)\sim
y^{1/4-it}(1-y)^{it}V(m,-t,T)G\\
\times\int_{-\infty}^{\infty}rh(r)\frac{e^{-ir(2\alt\log r-\Lt(\alt))}}{(r\alt)^{1/3}}\left(\frac{\alt^2\hat{\zeta}(y)}{y-1+\alt^2}\right)^{1/4}
Ai(-(2r\alt)^{2/3}\hat{\zeta}(y))dv,
\end{multline}
\begin{multline}\label{Ileq2Hinf as4}
I\left(\frac{n}{m},\rho;H_{\infty}\right)\sim
y^{1/4-it}(1-y)^{it}V(m,t,T)G\\\times
\int_{-\infty}^{\infty}rh(r)
e^{2iGv\log\frac{1+\alt_0}{1-\alt_0}}
\frac{e^{-ir(2\alt\log r-\Lt(\alt))}}{(r\alt)^{1/3}}\left(\frac{\alt^2\hat{\zeta}(y)}{y-1+\alt^2}\right)^{1/4}Ai(-(2r\alt)^{2/3}\hat{\zeta}(y))
dv.
\end{multline}
 Note that
\begin{equation}\label{Lt alt def}
2\alt\log r-\Lt(\alt)=
2\alt\log r+f(\alt),
\end{equation}
where $f(\alt)$ is given by \eqref{Linfty/Linfty f def}. Therefore, arguing as in Lemma \ref{lem: Linf/Linf y} we prove that
\begin{multline}\label{Ileq2H0inf as6}
\frac{I\left(\frac{n}{m},\rho;h\right)}{y^{1/4-it}(1-y)^{it}}\sim
V(m,\sgh(h)t,T)G\\\times
\int_{-\infty}^{\infty}e^{i\sgh(h)Gv\log\frac{1+\alt_0}{1-\alt_0}}
\frac{rh(r)}{(r\alt)^{1/3}}\left(\frac{\alt^2\hat{\zeta}(y)}{y-1+\alt^2}\right)^{1/4}Ai(-(2r\alt)^{2/3}\hat{\zeta}(y))dv,
\end{multline}
where
\begin{equation}\label{epsilon(h) def}
\sgh(h)=-1\quad\hbox{if}\quad h=H_0,\quad
\sgh(h)=-1\quad\hbox{if}\quad  h=H_{\infty}.
\end{equation}
In the case \eqref{y>1-alt0^2+  condition} using  \cite[9.7.8]{HMF} (see also \eqref{2f1 1/4 1/4 1/2 asympt y>1-alt^2})  one has
\begin{equation}\label{Ileq2H0inf as7}
\frac{I\left(\frac{n}{m},\rho;h\right)}{y^{1/4-it}(1-y)^{it}}\sim
V(m,\sgh(h)t,T)G
\int_{-\infty}^{\infty}e^{i\sgh(h)Gv\log\frac{1+\alt_0}{1-\alt_0}}
\frac{h(r)\sqrt{r}}{(y-1+\alt^2)^{1/4}}
\sum_{\pm}e^{\pm2ir\Aa_1(\alt,y)}dv.
\end{equation}
Substituting \eqref{Aa1 def} one has
\begin{equation}\label{Ileq2H0inf as7}
\frac{I\left(\frac{n}{m},\rho;h\right)}{y^{1/4-it}(1-y)^{it}}\sim
\sum_{\pm}\frac{V(m,\sgh(h)t,T)T^{1/2}G}{(1-y)^{\pm it}}
\int_{-\infty}^{\infty}e^{i\sgh(h)Gv\log\frac{1+\alt_0}{1-\alt_0}}
\frac{e^{-v^2}e^{\pm2i(T+Gv)\Aa_1(\alt)}}{(y-1+\alt^2)^{1/4}}dv,
\end{equation}
where
\begin{multline}\label{Aa1 def2}
\Aa_1(\alt)=
\alt\log\left(\alt\sqrt{y}+\sqrt{y-1+\alt^2}\right)-\log\left(\sqrt{y}+\sqrt{y-1+\alt^2}\right)+
\frac{1-\alt}{2}\log(1-\alt^2).
\end{multline}
For $x=Gv/T$ the following identity holds:
\begin{equation}
(y-1+\alt^2)^{-1/4}=\alt^{-1/2}\left(1+\frac{y-1}{\alt_0^2}\left(1+x\right)^2\right)^{-1/4}.
\end{equation}
In order to expand this function into a Taylor series around the point $x=0$, it is required to have the estimate
\begin{equation}
\frac{x}{1+\frac{y-1}{\alt_0^2}}=\frac{x\alt_0^2}{y-1+\alt_0^2}\ll T^{-\delta/2}.
\end{equation}
Applying \eqref{y>1-alt0^2+  condition} we conclude that
\begin{equation}
\frac{x}{1+\frac{y-1}{\alt_0^2}}\ll x\alt_0^{2/3}T^{2/3-2\delta}\ll Gt^{2/3}T^{-1-\delta} \ll T^{-\delta/2}
\end{equation}
as long as $G\ll T^{1+\delta/2}/t^{2/3}$. Assuming that this condition is satisfied,  one has
\begin{equation}\label{Ileq2H0inf as8}
\frac{I\left(\frac{n}{m},\rho;h\right)}{y^{1/4-it}(1-y)^{it}}\sim
\sum_{\pm}\frac{V(m,\sgh(h)t,T)T^{1/2}G}{(1-y)^{\pm it}(y-1+\alt_0^2)^{1/4}}
\int_{-\infty}^{\infty}e^{i\sgh(h)Gv\log\frac{1+\alt_0}{1-\alt_0}}
e^{-v^2}e^{\pm2i(T+Gv)\Aa_1(\alt)}dv.
\end{equation}
Recall that $\alt=\alt(x)=\alt_0/(1+x),$ $x=Gv/T$.  Let
\begin{equation}\label{F(x) def}
F(x):=(1+x)\Aa_1\left(\frac{\alt_0}{1+x}\right).
\end{equation}
Then the following Taylor series expansion around $x=0$ takes place:
\begin{equation}\label{Aa1 to F Taylor}
(T+Gv)\Aa_1(\alt)=TF(0)+F'(0)Gv+\sum_{k=2}^{\infty}\frac{F^{(k)}(0)}{k!}\frac{(Gv)^k}{T^{k-1}}.
\end{equation}
Here $F(0)=\Aa_1(\alt_0)$ and
\begin{equation}
F'(0)=\Aa_1(\alt_0)-\alt_0\frac{d}{d\alt}\Aa_1(\alt_0),\quad
F''(0)=\alt^2_0\frac{d^2}{d\alt^2}\Aa_1(\alt_0),\quad
F^{(k)}(0)=\sum_{j=2}^kc_j\alt^j_0\frac{d^j}{d\alt^j}\Aa_1(\alt_0),
\end{equation}
where $c_j$ are some absolute constants. Denote by $f_j(\alt)$ the $j$th summand in \eqref{Aa1 def2} so that $\Aa_1(\alt)=f_1(\alt)-f_2(\alt)+f_3(\alt).$
Consider
\begin{multline}\label{f1'}
f_1'(\alt)=\log\left(\alt\sqrt{y}+\sqrt{y-1+\alt^2}\right)+\frac{\alt(\alt+\sqrt{y}\sqrt{y-1+\alt^2})}{(\alt\sqrt{y}+\sqrt{y-1+\alt^2})\sqrt{y-1+\alt^2}}=\\=
\log\left(\alt\sqrt{y}+\sqrt{y-1+\alt^2}\right)+
\left(\frac{\alt}{\sqrt{y}}+\sqrt{y-1+\alt^2}\right)\\\times
\left(\frac{1}{\sqrt{y-1+\alt^2}}-\frac{1}{\alt\sqrt{y}+\sqrt{y-1+\alt^2}}\right).
\end{multline}
To evaluate the derivative of the product of two brackets on  the right-hand side of \eqref{f1'} we apply the product rule, then open the brackets, and finally use the relations
\begin{equation}
\frac{1}{\sqrt{y}\sqrt{y-1+\alt^2}}-\frac{1}{\sqrt{y}(\alt\sqrt{y}+\sqrt{y-1+\alt^2})}-
\frac{\alt}{\sqrt{y-1+\alt^2}(\alt\sqrt{y}+\sqrt{y-1+\alt^2})}=0,
\end{equation}
\begin{multline}
\frac{\alt(\alt+\sqrt{y}\sqrt{y-1+\alt^2})}{\sqrt{y}\sqrt{y-1+\alt^2}(\alt\sqrt{y}+\sqrt{y-1+\alt^2})^2}+
\frac{\alt+\sqrt{y}\sqrt{y-1+\alt^2}}{(\alt\sqrt{y}+\sqrt{y-1+\alt^2})^2}=\\=
\frac{(\alt+\sqrt{y}\sqrt{y-1+\alt^2})^2}{\sqrt{y}\sqrt{y-1+\alt^2}(\alt\sqrt{y}+\sqrt{y-1+\alt^2})^2},
\end{multline}
showing that
\begin{multline}
f_1''(\alt)=\frac{\alt+\sqrt{y}\sqrt{y-1+\alt^2}}{(\alt\sqrt{y}+\sqrt{y-1+\alt^2})\sqrt{y-1+\alt^2}}
-\frac{\alt^2}{(y-1+\alt^2)^{3/2}\sqrt{y}}+\\+
\frac{(\alt+\sqrt{y}\sqrt{y-1+\alt^2})^2}{\sqrt{y}(\alt\sqrt{y}+\sqrt{y-1+\alt^2})^2\sqrt{y-1+\alt^2}}.
\end{multline}
Note that while evaluating $f_1''(\alt)$ the terms $\frac{\pm\alt}{y-1+\alt^2}$ were cancelled out.
One has
\begin{equation}
f_2'(\alt)=\frac{\alt(y-1+\alt^2)^{-1/2}}{(\sqrt{y}+\sqrt{y-1+\alt^2})}=
\frac{\alt}{\sqrt{y}\sqrt{y-1+\alt^2}}-\frac{\alt}{(\sqrt{y}+\sqrt{y-1+\alt^2})\sqrt{y}},
\end{equation}
\begin{equation}
f_2''(\alt)=\frac{(y-1+\alt^2)^{-1/2}}{(\sqrt{y}+\sqrt{y-1+\alt^2})}-\frac{\alt^2}{(y-1+\alt^2)^{3/2}\sqrt{y}}+
\frac{\alt^2(y-1+\alt^2)^{-1/2}}{\sqrt{y}(\sqrt{y}+\sqrt{y-1+\alt^2})^2}.
\end{equation}
Therefore,
\begin{multline}\label{F der0}
F'(0)=
-\log\left(\sqrt{y}+\sqrt{y-1+\alt_0^2}\right)+\frac{1}{2}\log(1-\alt_0^2)+\frac{\alt_0^2}{1+\alt_0}+\\+
\frac{\alt_0^2}{\sqrt{y-1+\alt_0^2}}\left(\frac{1}{\sqrt{y}+\sqrt{y-1+\alt_0^2}}-
\frac{\alt_0+\sqrt{y}\sqrt{y-1+\alt_0^2}}{\alt_0\sqrt{y}+\sqrt{y-1+\alt_0^2}}\right)=
-\log\left(\frac{\sqrt{y}+\sqrt{y-1+\alt_0^2}}{\sqrt{1-\alt^2_0}}\right)
\end{multline}
and
\begin{multline}
\frac{F''(0)}{\alt_0^2}=\frac{\alt_0}{1-\alt_0^2}-\frac{1}{(1+\alt_0)^2}+
\frac{\alt_0+\sqrt{y}\sqrt{y-1+\alt_0^2}}{(\alt_0\sqrt{y}+\sqrt{y-1+\alt_0^2})\sqrt{y-1+\alt_0^2}}
+\\+
\frac{(\alt_0+\sqrt{y}\sqrt{y-1+\alt_0^2})^2}{\sqrt{y}(\alt_0\sqrt{y}+\sqrt{y-1+\alt_0^2})^2\sqrt{y-1+\alt_0^2}}-
\frac{(y-1+\alt_0^2)^{-1/2}}{(\sqrt{y}+\sqrt{y-1+\alt_0^2})}-
\frac{\alt_0^2(y-1+\alt_0^2)^{-1/2}}{\sqrt{y}(\sqrt{y}+\sqrt{y-1+\alt_0^2})^2}.
\end{multline}
In view of \eqref{y>1-alt0^2+  condition} the following estimates hold:
\begin{equation}
F''(0)\ll\frac{\alt_0^2}{\sqrt{y-1+\alt_0^2}},\quad
F^{(k)}(0)\ll\frac{\alt_0^{2k-2}}{(y-1+\alt_0^2)^{k-3/2}},
\end{equation}
and consequently,
\begin{equation}\label{F derk}
\frac{F^{(k)}(0)(Gv)^k}{T^{k-1}}\ll T^{3\delta/2}\left(\frac{Gt^{2/3}\log^2T}{T^{1+\delta}}\right)^k\ll T^{-k\delta_1}
\end{equation}
provided that  $G\ll T^{1+\delta/4}/t^{2/3}$.
Therefore, one can rewrite \eqref{Ileq2H0inf as8} as
\begin{equation}\label{Ileq2H0inf as9}
\frac{I\left(\frac{n}{m},\rho;h\right)}{y^{1/4-it}(1-y)^{it}}\sim
\sum_{\pm}\frac{V(m,\sgh(h)t,T)T^{1/2}G}{(1-y)^{\pm it}(y-1+\alt_0^2)^{1/4}}
e^{\pm2iT\Aa_1(\alt_0)}\int_{-\infty}^{\infty}e^{2iGv\left(\sgh(h)\log\frac{\sqrt{1+\alt_0}}{\sqrt{1-\alt_0}}\pm F'(0)\right)}
e^{-v^2}dv.
\end{equation}
Applying \eqref{exp integral1} we show that
\begin{equation}\label{Ileq2H0inf as10}
\frac{I\left(\frac{n}{m},\rho;h\right)}{y^{1/4-it}(1-y)^{it}}\sim
\sum_{\pm}\frac{V(m,\sgh(h)t,T)T^{1/2}G}{(1-y)^{\pm it}(y-1+\alt_0^2)^{1/4}}
e^{\pm2iT\Aa_1(\alt_0)}
e^{-G^2\left(\sgh(h)\log\frac{\sqrt{1+\alt_0}}{\sqrt{1-\alt_0}}\pm F'(0)\right)^2}.
\end{equation}
Using \eqref{F der0} we obtain
\begin{equation}\label{Ileq2H0inf as110}
\epsilon(H_0)\log\frac{\sqrt{1+\alt_0}}{\sqrt{1-\alt_0}}\pm F'(0)=-\log\frac{\sqrt{1+\alt_0}}{\sqrt{1-\alt_0}}
\mp\log\left(\frac{\sqrt{y}+\sqrt{y-1+\alt_0^2}}{\sqrt{1-\alt^2_0}}\right)=
\mp\log\left(\frac{\sqrt{y}+\sqrt{y-1+\alt_0^2}}{1\mp\alt_0}\right),
\end{equation}
\begin{equation}\label{Ileq2H0inf as11inf}
\epsilon(H_{\infty})\log\frac{\sqrt{1+\alt_0}}{\sqrt{1-\alt_0}}\pm F'(0)=\log\frac{\sqrt{1+\alt_0}}{\sqrt{1-\alt_0}}
\mp\log\left(\frac{\sqrt{y}+\sqrt{y-1+\alt_0^2}}{\sqrt{1-\alt^2_0}}\right)=
\mp\log\left(\frac{\sqrt{y}+\sqrt{y-1+\alt_0^2}}{1\pm\alt_0}\right).
\end{equation}
Let $y=1-\alt_0^2+\beta^2$. It follows from \eqref{y>1-alt0^2+  condition} that $\beta>\frac{\alt_0 T^{\delta}}{t^{1/3}}$.  As a result,
\begin{equation}\label{Ileq2H0inf as120inf}
\log\left(\frac{\sqrt{y}+\sqrt{y-1+\alt_0^2}}{1\pm\alt_0}\right)=
\log\left(\sqrt{1-\alt_0^2+\beta^2}+\beta\right)\mp\alt_0+O(\alt_0^2)=
\beta\mp\alt_0+O(\alt_0^2).
\end{equation}
Therefore, for $G\ll T^{1+\delta/4}/t^{2/3}$, $G\alt_0\gg T^{\delta}$ one has $$e^{-G^2\log^2\left(\frac{\sqrt{y}+\sqrt{y-1+\alt_0^2}}{1-\alt_0}\right)}\ll T^{-A}.$$
Thus
\begin{equation}\label{Ileq2H0 as13}
\frac{I\left(\frac{n}{m},\rho;H_0\right)}{y^{1/4-it}(1-y)^{it}}\sim
\frac{V(m,-t,T)T^{1/2}G}{(1-y)^{-it}(y-1+\alt_0^2)^{1/4}}
e^{-2iT\Aa_1(\alt_0)}
e^{-G^2\log^2\left(\frac{\sqrt{y}+\sqrt{y-1+\alt_0^2}}{1+\alt_0}\right)}.
\end{equation}
\begin{equation}\label{Ileq2Hinf as13}
\frac{I\left(\frac{n}{m},\rho;H_{\infty}\right)}{y^{1/4-it}(1-y)^{it}}\sim
\frac{V(m,t,T)T^{1/2}G}{(1-y)^{it}(y-1+\alt_0^2)^{1/4}}
e^{2iT\Aa_1(\alt_0)}
e^{-G^2\log^2\left(\frac{\sqrt{y}+\sqrt{y-1+\alt_0^2}}{1+\alt_0}\right)}.
\end{equation}
We are left to consider the interval complementary to \eqref{y>1-alt0^2+  condition}, \eqref{y<1-alt0^2+  condition}, namely
\begin{equation}\label{1-alt0^2-<y<1-alt0^2+  condition}
1-\alt_0^2-\frac{\alt_0^{4/3}}{T^{2/3-2\delta}}<y<1-\alt_0^2+\frac{\alt_0^{4/3}}{T^{2/3-2\delta}}.
\end{equation}
On this interval the argument of the Airy function in \eqref{Ileq2H0inf as6} is bounded by $T^{\epsilon}$. Heuristically, this means that the Airy function does not oscillate (or at least its oscillation does not affect the integral) and can be replaced by one. Doing so, one obtains  the integral (see \eqref{Ileq2H0inf as6}):
\begin{equation}
\int_{-\infty}^{\infty}e^{i\sgh(h)Gv\log\frac{1+\alt_0}{1-\alt_0}}dv=e^{-(G/2)^2\log^2\frac{1+\alt_0}{1-\alt_0}}\ll T^{-A}
\end{equation}
for $G\alt_0\gg T^{\delta}$.  In order to prove this rigorously, we consider several cases. We will work out in all details only the case $y>1-\alt_0^2$ since the case $y<1-\alt_0^2$ can be treated analogously. It  follows from \eqref{Aa1 Taylor} that for $G\ll T^{1-2\delta}t^{-2/3}$ and
\begin{equation}\label{1-alt0^2-<y<1-alt0^2+  condition0}
|y-1+\alt_0^2|<\frac{\alt_0^{2}}{t^{2/3}T^{\delta}},
\end{equation}
one has (we use the notation $y:=1-\alt_0^2+\beta^2$, $\beta<\frac{\alt_0}{T^{\delta/2}t^{1/3}}$)
\begin{multline}
(T\Aa_1(\alt,y))^{2/3}\ll \frac{T^{2/3}(y-1+\alt^2)}{\alt_0^{4/3}}\ll
\frac{T^{2/3}(\beta^2+\alt^2-\alt_0^2)}{\alt_0^{4/3}}\ll\\\ll
\frac{T^{2/3}(\beta^2+\alt_0^2GT^{-1+\epsilon})}{\alt_0^{4/3}}\ll
\frac{T^{2/3-\delta}\alt_0^{2/3}}{t^{2/3}}\ll T^{-\delta}.
\end{multline}
Therefore, the argument of the Airy function is less than $T^{-\delta}$, and  one can represent the Airy function via its Taylor series representation \cite[9.4.1]{HMF}.

We are left to consider the case
\begin{equation}\label{1-alt0^2-<y<1-alt0^2+  condition1}
\frac{\alt_0^{2}}{t^{2/3}T^{\delta}}<y-1+\alt_0^2<\frac{\alt_0^{4/3}}{T^{2/3-2\delta}}.
\end{equation}
Under these conditions,  we represent the argument of the Airy function in \eqref{Ileq2H0inf as6} as a function of $x=Gv/T$ and again expand it in the  Taylor series at the point $x=0$. Applying \eqref{zeta(y) def}, \eqref{Aa1 def},  \eqref{Aa1 def2} and  \eqref{F(x) def}, we show that the argument  can be written as
\begin{equation}\label{Ileq2 est1.0}
(3r\Aa_1(\alt,y))^{2/3}=\left(-\frac{3t}{2}\log(1-y)+3(T+Gv)\Aa_1(\alt)\right)^{2/3}.
\end{equation}
Applying \eqref{Aa1 to F Taylor}, \eqref{F der0}, \eqref{F derk} and arguing as in \eqref{Ileq2H0inf as120inf},  we obtain
\begin{equation}\label{Ileq2 est1}
(3r\Aa_1(\alt,y))^{2/3}=
\left(3T\Aa_1(\alt_0,y)+O(G^{1+\epsilon}\sqrt{y-1+\alt_0^2})\right)^{2/3}.
\end{equation}
For the main term  of \eqref{Ileq2 est1} we have the following lower bound $$T\Aa_1(\alt_0,y)>T^{-3\delta/2},$$ which is the consequence of \eqref{Aa1 Taylor} and \eqref{1-alt0^2-<y<1-alt0^2+  condition1}. Furthermore, the error term in \eqref{Ileq2 est1} is less than $T^{-3\delta}$ if $G\ll T^{1-4\delta}t^{-2/3}$. Therefore, the Airy function in \eqref{Ileq2H0inf as6} looks like $$Ai((M+E)^{2/3})=Ai(M^{2/3}+Y),$$ where
\begin{equation}
M=3T\Aa_1(\alt_0,y),\quad E=O(G^{1+\epsilon}\sqrt{y-1+\alt^2}),\quad
Y=O(EM^{-1/3})
\end{equation}
since $E\ll T^{-\delta}M$, and $Ai((M+E)^{2/3})$ can be expanded into the Taylor series as follows:
\begin{equation}
Ai((M+E)^{2/3})=Ai(M^{2/3}+Y)=\sum_{k=0}^{\infty}Ai^{(k)}(M^{2/3})\frac{Y^k}{k!}.
\end{equation}
Since $Ai^{(k)}(z)\ll \max(z^{k/2},1)Ai(z)$ and  $M^{2/3}Y^2\ll T^{-\delta}$, the series above yields the asymptotic expansion.
This means that we have proved the following result.
\begin{prop}\label{prop: S1 est1}
For $T^{1+\epsilon}/t\ll G\ll T^{1-\epsilon}/t^{2/3}$ one has
\begin{equation}\label{S1 est}
S_1(t)\ll T^{1/2}G
\sum_{h=H_0,H_{\infty}}\sum_{m\ll T^{1+\epsilon}\sqrt{t}}\sum_{0<2m-n\ll m\alt_0G^{-1+\epsilon}}\frac{\mathscr{L}_{n^2-4m^2}(1/2+2it)}{n^{1/2-2it}m^{1/2+2i\sgh(h)t}}
\LB\left(\sgh(h),\frac{n^2}{4m^2}\right),
\end{equation}
where $\sgh(h)=-1$ if $h=H_0$, $\sgh(h)=1$ for $h=H_{\infty}$ and
\begin{equation}\label{LB def}
\LB\left(\sgh(h),y\right)=
\frac{V(m,\sgh(h)t,T)(1-y)^{i(1-\sgh(h))t}y^{1/4-it}}{(y-1+\alt_0^2)^{1/4}}
e^{2iT\sgh(h)\Aa_1(\alt_0)}
e^{-G^2\log^2\left(\frac{\sqrt{y}+\sqrt{y-1+\alt_0^2}}{1+\alt_0}\right)}
\end{equation}
with $\Aa_1(\alt_0)$ defined as in \eqref{Aa1 def2}.
\end{prop}

\subsection{Preliminary estimate on $S_2(t)$}\label{subsec:S2}
Our goal is to estimate the term $S_2(t)$ using Proposition \ref{prop: 2F1 asympt1}. We start with truncating the summation over $n$ in \eqref{S2 def}.
\begin{lem}\label{Lem I x>>2 est}
For $h=H_0$ or $h=H_{\infty}$ and $x\gg \max(T^{1+\delta}m^{-1}\sqrt{1+|t|},1)$ one has
\begin{equation}\label{I x>>2 est}
I\left(x,1/2+2it;h\right)\ll T^{-A}.
\end{equation}
\end{lem}
\begin{proof}
Using the integral representation \cite[15.6.1]{HMF} for the hypergeometric function in \eqref{integralIgeq2}, we obtain
\begin{multline}\label{Igeq2 est1}
I(x,1/2+2it;h)\ll\int_{-\infty}^{\infty}\frac{rh(r)\sin\left( \pi(1/4-ir+it)\right)}{\cosh(\pi r)}
\frac{\Gamma(1/4+ir-it)}{\Gamma(1/4+ir+it)}\\ \times
\left(\frac{2}{x}\right)^{2ir}
\int_0^1\frac{y^{-1/4+ir-it}(1-y)^{-3/4+ir+it}}{(1-4y/x^2)^{1/4+ir-it}}dydr.
\end{multline}
Let $x\gg1$. Moving the line of integration to $\Im{r}=-N$,  estimating the Gamma-factors and using \eqref{V(t-iN)est}, we prove that
\begin{multline}\label{Igeq2 est1}
I(x,1/2+2it;h)\ll\int_{T-G^{1+\epsilon}}^{T+G^{1+\epsilon}}r
\frac{(1+|t|)^N(1+|r|)^{2N+\epsilon}}{m^{2N}}\\ \times
\left(\frac{2}{x}\right)^{2N}
\int_0^1\frac{y^{-1/4+N}(1-y)^{-3/4+N}}{(1-4y/x^2)^{1/4+N}}dydr
\ll T^{-A}
\end{multline}
for $mx\gg T^{1+\delta}\sqrt{1+|t|}$.
\end{proof}
Lemma \ref{Lem I x>>2 est} shows that the part of \eqref{S2 def} with $n\gg T^{1+\delta}\sqrt{t}$ is negligibly small. Therefore, from now on we may assume that
\begin{equation}\label{S2 def2}
S_2(t)=
\sum_{m\ll T^{1+\epsilon}\sqrt{t}}\sum_{2m<n\ll T^{1+\delta}\sqrt{t}}\frac{\mathscr{L}_{n^2-4m^2}(1/2+2it)}{n^{1/2-2it}}
\left(\frac{I\left( \frac{n}{m},1/2+2it;H_0\right)}{m^{1/2-2it}}+\frac{I\left( \frac{n}{m},1/2+2it;H_{\infty}\right)}{m^{1/2+2it}}\right).
\end{equation}
Since all sums in \eqref{S2 def2}  are now finite it is sufficient to consider only the contribution of main terms in asymptotic expansions. Furthermore, we are not interested in keeping track on various constants independent of summation variables. To simplify the notation, we use the sign $\sim$ to indicate the main term of an asymptotic expansion  up to such constants.

Let $z:=4m^2/n^2$. Substituting \eqref{2F1 main asymptI} to \eqref{integralIgeq2}, one has
\begin{multline}\label{Igeq2 as1}
I(\frac{n}{m},\rho;h)\sim\int_{-\infty}^{\infty}\frac{rh(r)z^{ir}}{\cosh(\pi r)}
\frac{\Gamma(1/4+ir(1-\alt))\Gamma(3/4+ir(1-\alt))}{\Gamma(1+2ir)}\\\times\sin\left( \pi(1/4-ir(1-\alt))\right)
\frac{\exp\left(2ir\Lo(\alt,r,z)\right)}{(1-(1-\alt^2)z)^{1/4}}dr.
\end{multline}
The Stirling formula  \eqref{Stirling2} implies that
\begin{equation}
\frac{\Gamma(1/4+ir(1-\alt))\Gamma(3/4+ir(1-\alt))}{\Gamma(1+2ir)\cosh(\pi r)}\sin\left( \pi(1/4-ir(1-\alt))\right)\sim
\frac{e^{2ir\left(-\alt\log r+(1-\alt)\log(1-\alt)+\alt-\log2\right)}}{r^{1/2}}.
\end{equation}
Then for $r=T+Gy$, $|y|\ll \log^2 T$, $G\ll T^{1-\epsilon}/\sqrt{t}$ and $\alt_0=t/T\ll T^{-\delta}$  arguing in the same way as in Lemma \ref{lem: Linf/Linf y} we obtain
\begin{equation}\label{GamGamsin/Gamcosh y}
\frac{\Gamma(1/4+ir(1-\alt))\Gamma(3/4+ir(1-\alt))}{\Gamma(1+2ir)\cosh(\pi r)}\sin\left( \pi(1/4-ir(1-\alt))\right)\sim
e^{2iGY\left(\log(1-\alt_0)-\log2\right)}T^{-1/2}.
\end{equation}
Making in \eqref{Igeq2 as1} the change of variables $r:=T+Gy$, and then applying  \eqref{GamGamsin/Gamcosh y}, \eqref{2F1 main asymptIy} and
\eqref{Vapprox y=0}, we infer that for $G\ll T^{1/2-\epsilon}$
\begin{multline}\label{Igeq2 H0as1}
I(\frac{n}{m},\rho;H_0)\sim \frac{T^{1/2}G V(m,-t,T)}{(1-(1-\alt_0^2)z)^{1/4}}
e^{2iT\La(\alt_0,z)}
\int_{-\infty}^{\infty} e^{-y^2}
e^{2iGy\left(\log(1-\alt_0)-\log(1+\sqrt{1-(1-\alt_0^2)z})+\log\sqrt{z}\right)}dy,
\end{multline}
where
\begin{equation}\label{La def}
\La(\alt_0,z)=-\log(1+\sqrt{1-(1-\alt_0^2)z})+\alt_0\log(\alt_0+\sqrt{1-(1-\alt_0^2)z})+\log\sqrt{z}.
\end{equation}

Similarly, but also using \eqref{Linfty/Linfty}, one has
\begin{multline}\label{Igeq2 Hinfas1}
I(\frac{n}{m},\rho;H_{\infty})\sim \frac{T^{1/2}G V(m,t,T)}{(1-(1-\alt_0^2)z)^{1/4}}
e^{2iT\La(\alt_0,z)}
\int_{-\infty}^{\infty} e^{-y^2}
e^{2iGy\left(\log(1+\alt_0)-\log(1+\sqrt{1-(1-\alt_0^2)z})+\log\sqrt{z}\right)}dy.
\end{multline}
Applying \eqref{exp integral1}, we conclude that
\begin{equation}\label{Igeq2 H0 as2}
I(\frac{n}{m},\rho;H_0)\sim \frac{T^{1/2}G V(m,-t,T)}{(1-(1-\alt_0^2)z)^{1/4}}
e^{2iT\La(\alt_0,z)}
e^{-G^2\log^2\frac{(1-\alt_0)\sqrt{z}}{1+\sqrt{1-(1-\alt_0^2)z}}},
\end{equation}
\begin{equation}\label{Igeq2 Hinf as2}
I(\frac{n}{m},\rho;H_{\infty})\sim \frac{T^{1/2}G V(m,t,T)}{(1-(1-\alt_0^2)z)^{1/4}}
e^{2iT\La(\alt_0,z)}
e^{-G^2\log^2\frac{(1+\alt_0)\sqrt{z}}{1+\sqrt{1-(1-\alt_0^2)z}}}.
\end{equation}
Therefore, for $G\ll T^{1/2-\epsilon}$ one has
\begin{equation}\label{S2 def3}
S_2(t)\ll T^{1/2}G
\sum_{m\ll T^{1+\epsilon}\sqrt{t}}\sum_{2m<n\ll T^{1+\delta}\sqrt{t}}\frac{\mathscr{L}_{n^2-4m^2}(1/2+2it)}{n^{1/2-2it}}\frac{V(m,\pm t,T)}{m^{1/2\pm 2it}}
\LA_{\pm}\left(\frac{4m^2}{n^2}\right),
\end{equation}
where
\begin{equation}\label{S2 LAdef}
\LA_{\pm}\left(z\right)=
\frac{e^{2iT\La(\alt_0,z)}}{(1-(1-\alt_0^2)z)^{1/4}}
\exp\left(-G^2\log^2\frac{(1\pm\alt_0)\sqrt{z}}{1+\sqrt{1-(1-\alt_0^2)z}}\right).
\end{equation}

Let $$f_{\pm}(z):=\log\frac{(1\pm\alt_0)\sqrt{z}}{1+\sqrt{1-(1-\alt_0^2)z}}.$$
The right-hand side of \eqref{S2 def3} is negligibly small unless $|f_{\pm}(z)|\ll G^{-1+\epsilon}.$
Since $$f'_{\pm}(z)=\frac{1}{2z\sqrt{1-(1-\alt_0^2)z}},$$ these functions are increasing on the interval $0<z<1$. Consequently,
\begin{equation}
f_{-}(z)<f_{-}(1)=\log\frac{1-\alt_0}{1+\alt_0}\ll -\alt_0,
\end{equation}
and for $G\alt_0\gg G^{\epsilon}$ one has $|f_{-}(z)|\gg G^{-1+\epsilon}$. As a result, $\LA_{-}\left(z\right)\ll T^{-A}.$ In the plus case, we consider separately  two cases: $0<z<1-\alt_0^2$ and $1-\alt_0^2<z<1$. In the first case, one has
\begin{equation}
f_{+}(z)<f_{-}(1-\alt_0^2)=(1-\sqrt{2})\alt_0+O(\alt_0^2)\ll -\alt_0.
\end{equation}
Therefore, for $G\alt_0\gg G^{\epsilon}$, $0<z<1-\alt_0^2$ the following estimate holds: $\LA_{+}\left(z\right)\ll T^{-A}.$
In the case when $1-\alt_0^2<z<1$, one has
\begin{equation}\label{f+ z=1 est}
g_{+}(z):=f_{+}(z)+\frac{1-z}{4\alt_0}<0,
\end{equation}
and therefore, for $G\alt_0\gg G^{\epsilon}$  the estimate $\LA_{+}\left(\frac{4m^2}{n^2}\right)\ll T^{-A}$ holds unless $|n-2m|\ll m\alt_0G^{-1+\epsilon}$.
Since $g_{+}(1)=0$ to prove \eqref{f+ z=1 est} it is enough to show that $g'_{+}(z)>0$ for $1-\alt_0^2<z<1$.  Note that
\begin{equation}
g''_{+}(z)=\frac{3(1-\alt_0^2)z-2}{2z^2(1-(1-\alt_0^2)z)^{3/2}}>0.
\end{equation}
Therefore, we are left to verify that $g'_{+}(1-\alt_0^2)>0$. This follows from straightforward computations:
\begin{equation}
g'_{+}(1-\alt_0^2)=\frac{1}{2(1-\alt_0^2)(1-(1-\alt_0^2)^2)^{1/2}}-\frac{1}{4\alt_0}=\frac{2-(1-\alt_0^2)(2-\alt_0^2)^{1/2}}{4\alt_0(1-\alt_0^2)(2-\alt_0^2)^{1/2}}>0.
\end{equation}

Finally, we prove the following result.


\begin{prop}
For $G\ll T^{1/2-\epsilon},$ $\alt_0\gg G^{-1+\epsilon}$ one has
\begin{equation}\label{S2 def4}
S_2(t)\ll T^{1/2}G
\sum_{m\ll T^{1+\epsilon}\sqrt{t}}\sum_{0<n-2m\ll m\alt_0G^{-1+\epsilon}}\frac{\mathscr{L}_{n^2-4m^2}(1/2+2it)}{n^{1/2-2it}}\frac{V(m,t,T)}{m^{1/2+2it}}
\LA_{+}\left(\frac{4m^2}{n^2}\right).
\end{equation}
\end{prop}
\begin{rem}
According to \eqref{LB def} and \eqref{S2 LAdef}
\begin{equation}\label{LA=LB}
V(m,t,T)\LA_{+}\left(\frac{1}{y}\right)=\LB\left(\sgh(H_{\infty}),y\right)e^{-iT(1-\alt_0)\log(1-\alt_0^2)}.
\end{equation}
\end{rem}


\subsection{Preparing to use the Voronoi formula}\label{subsec:S2 beforeVoron}
In this section, we transform the expressions on the right-hand sides of \eqref{S1 est} and \eqref{S2 def4} following the strategy used in \cite[Sec. 4]{BFsym2} in order to obtain a form suitable for the application of the Voronoi summation formula. First, we substitute \eqref{S2 LAdef} and \eqref{La def} to \eqref{S2 def4}, showing that
\begin{equation}\label{S2 def5}
S_2(t)\ll T^{1/2}G
\sum_{m\ll T^{1+\epsilon}\sqrt{t}}\sum_{0<n-2m\ll m\alt_0G^{-1+\epsilon}}\frac{
\Zag_{n^2-4m^2}(1/2+2it)V(m,t,T)}{m\left(\alt^2-1+n^2/(4m^2)\right)^{1/4}}
\LE\left(T,\frac{n}{2m}\right),
\end{equation}
\begin{equation}\label{LE def}
\LE(T,y)=\exp\left(-G^2\log^2\frac{y+\sqrt{y^2-1+\alt_0^2}}{1+\alt_0}\right)e^{2iT\Le(y)},
\end{equation}
\begin{equation}\label{Le def}
\Le(y)=\alt_0\log(y\alt_0+\sqrt{y^2-1+\alt_0^2})-\log(y+\sqrt{y^2-1+\alt_0^2}).
\end{equation}
For the sake of simplicity, we will write $\LE(y)$ instead of $\LE(T,y)$. Now we make the following change of variables in \eqref{S2 def5}:
\begin{equation}\label{change of variables1}
n-2m=q,\quad
n+2m=\rr.
\end{equation}
Doing so, we obtain a congruence condition  $\rr\equiv q\pmod{4}$, which can be rewritten  as in \cite[(4.11)]{BFsym2}. This results in four sums. All these sums can be treated analogously to each other (see \cite[Remark 4.1, 6.1]{BFsym2}). Therefore, as in \cite[Sec. 4]{BFsym2}  it is sufficient to study only one sum. For example, let us consider
\begin{equation}\label{S2 def6}
S_2(t)\ll T^{1/2}G
\sum_{\rr\ll T^{1+\epsilon}\sqrt{t}}\sum_{0<q\ll \rr\alt_0G^{-1+\epsilon}}\frac{
\Zag_{q\rr}(1/2+2it)V\left(\frac{\rr-q}{4},t,T\right)}{(\rr-q)^{1/2}\left(\alt^2(\rr-q)^{2}+4rq\right)^{1/4}}
\LE\left(\frac{\rr+q}{\rr-q}\right).
\end{equation}
Note that the sum over $q$ in \eqref{S2 def6} is much shorter than the one over $\rr$. Accordingly, for $\alt_0G\gg G^{\delta}$ using the Taylor series expansion  one has
\begin{equation}\label{S2 def6}
S_2(t)\ll \frac{T^{1/2}G}{\sqrt{\alt_0}}
\sum_{\rr\ll T^{1+\epsilon}\sqrt{t}}\sum_{0<q\ll \rr\alt_0G^{-1+\epsilon}}\frac{1}{\rr}\Zag_{q\rr}(1/2+2it)V\left(\frac{\rr-q}{4},t,T\right)
\LE\left(\frac{\rr+q}{\rr-q}\right).
\end{equation}
Next, we perform one more change of variables $\rr:=l/q$, obtaining
\begin{equation}\label{S2 def6}
S_2(t)\ll \frac{T^{1/2}G}{\sqrt{\alt_0}}
\sum_{q\ll t^{3/2}G^{-1+\epsilon}}\sum_{\substack{q^2G^{1-\epsilon}/\alt_0\ll l\ll qT^{1+\epsilon}\sqrt{t}\\l\equiv0\pmod{q}}}
\frac{q}{l}\Zag_{l}(1/2+2it)V\left(\frac{l-q^2}{4q},t,T\right)
\LE\left(\frac{l+q^2}{l-q^2}\right).
\end{equation}
Our next step is to estimate trivially the contribution of  $q\ll Q_0$. Performing a dyadic partition of the sum over $l$, applying the Cauchy-Schwartz inequality and  \eqref{LZag 2mom estimate}, we show that for $q^2G^{1-\epsilon}/\alt_0\ll L\ll qT^{1+\epsilon}\sqrt{t}$
\begin{equation}\label{S2 q<Q0}
\sum_{l\equiv0\pmod{q}}
\frac{U(l/L)}{l}\Zag_{l}(1/2+2it)V\left(\frac{l-q^2}{4q},t,T\right)
\LE\left(\frac{l+q^2}{l-q^2}\right)\ll\frac{(L+\sqrt{Lt})^{1/2}L^{\epsilon}}{\sqrt{qL}}.
\end{equation}
Therefore, the contribution of small $q$ is bounded by
\begin{equation}\label{S2 q<Q0 2}
\frac{T^{1/2+\epsilon}G}{\sqrt{\alt_0}}
\sum_{q\ll Q_0}\left(\left(\frac{\alt_0t}{G}\right)^{1/4}+\sqrt{q}\right)\ll
T^{1+\epsilon}G\left(\left(\frac{\alt_0}{tG}\right)^{1/4}Q_0+\frac{Q_0^{3/2}}{\sqrt{t}}\right)\ll T^{1-\epsilon}G,
\end{equation}
provided that
\begin{equation}\label{Q0 conditions}
Q_0\ll t^{1/3}T^{-\epsilon},\quad
Q_0\ll \left(\frac{tG}{\alt_0}\right)^{1/4}T^{-\epsilon}.
\end{equation}
For the rest of computations, we need
\begin{equation}\label{Q0 def}
Q_0=T^{\epsilon}\alt_0\sqrt{t},
\end{equation}
therefore, the conditions \eqref{Q0 conditions} are satisfied if
\begin{equation}\label{Q0 conditions GtT}
G\gg tT^{\epsilon}\alt_0^5, \quad
t\ll T^{6/7-\epsilon}.
\end{equation}
These conditions can be simplified to $t\ll T^{6/7-\epsilon}$ since in Proposition \ref{prop: S1 est1} $G\gg T^{1+\epsilon}/t$, and for
$t\ll T^{6/7-\epsilon}$  one has $T^{1+\epsilon}/t\gg  tT^{\epsilon}\alt_0^5 $. Therefore, from now we may assume that
\begin{equation}\label{Q0 conditions GtT2}
G\alt_0\gg T^{\epsilon}, \quad t\ll T^{6/7-\epsilon}.
\end{equation}
Under these conditions, one has
\begin{equation}\label{S2 def7}
S_2(t)\ll \frac{T^{1/2}G}{\sqrt{\alt_0}}
\sum_{Q_\ll q\ll t^{3/2}G^{-1+\epsilon}}\sum_{\substack{q^2G^{1-\epsilon}/\alt_0\ll l\ll qT^{1+\epsilon}\sqrt{t}\\l\equiv0\pmod{q}}}
\frac{q}{l}\Zag_{l}(1/2+2it)V\left(\frac{l-q^2}{4q},t,T\right)
\LE\left(\frac{l+q^2}{l-q^2}\right).
\end{equation}
\begin{rema}
The restriction $t\ll T^{6/7-\epsilon}$ is not sharp as it is a consequence of the non-optimal  estimate \eqref{S2 q<Q0}. If one can prove the following  generalization of
\eqref{LZag 2mom estimate}
\begin{equation}\label{LZag 2mom estimate AP}
\sum_{\substack{N<n< 2N\\ n\equiv0\pmod{q}}}|\Zag_n(1/2+it)|^2\ll\frac{1}{q}\left(N+\sqrt{N(1+|t|)}\right)\left(N(1+|t|)\right)^{\epsilon},
\end{equation}
then  instead of \eqref{S2 q<Q0 2} we obtain
\begin{equation}\label{S2 q<Q0 2 v2}
\frac{T^{1/2+\epsilon}G}{\sqrt{\alt_0}}
\sum_{q\ll Q_0}\left(\left(\frac{\alt_0t}{q^2G}\right)^{1/4}+1\right)\ll
T^{1+\epsilon}G\left(\left(\frac{\alt_0Q_0^2}{tG}\right)^{1/4}+\frac{Q_0}{\sqrt{t}}\right)\ll T^{1-\epsilon}G.
\end{equation}
This inequality holds for  $Q_0=T^{\epsilon}\alt_0\sqrt{t}$ if $t\ll T^{1-\epsilon_0}.$
\end{rema}
Expressing the congruence condition $l\equiv 0\pmod{q}$ in terms of additive harmonics
\begin{equation}\label{deltaq(l)}
\delta_q(l)=\frac{1}{q}\sum_{c|q}\sum_{\substack{a\pmod{c}\\(a,c)=1}}e\left(\frac{al}{c}\right),
\end{equation}
making the sum over $c$ an outer one and performing a dyadic partition, we obtain
\begin{equation}\label{S2 to S2(CQL)}
S_2(t)\ll \frac{T^{1/2}G}{\sqrt{\alt_0}}
\sum_{C\ll t^{3/2}G^{-1+\epsilon}}\sum_{Q_0/C\ll Q\ll t^{3/2}G^{-1+\epsilon}/C}
\sum_{Q^2C^2G^{1-\epsilon}/\alt_0\ll L\ll QCT^{1+\epsilon}\sqrt{t}}S_2(C,Q,L),
\end{equation}
where
\begin{equation}\label{S2(CQL) def}
S_2(C,Q,L)=
\sum_{c}U\left(\frac{c}{C}\right)\sum_{q}U\left(\frac{q}{Q}\right)
\sum_{\substack{a\pmod{c}\\(a,c)=1}}\sum_{l}
\frac{\Zag_{l}(1/2+2it)}{l^{1/2-it}}\phi_2(l)e\left(\frac{al}{c}\right),
\end{equation}
\begin{equation}\label{S2 phi2 def}
\phi_2(l)=\frac{U(l/L)}{l^{1/2+it}}V\left(\frac{l-(qc)^2}{4qc},t,T\right)
\LE\left(\frac{l+(qc)^2}{l-(qc)^2}\right)
\end{equation}
and $U(x)\neq0$ only if $1/2<x<2.$

Now we are going to transform $S_1(t)$ similarly to $S_2(t)$.  It follows from
Proposition \ref{prop: S1 est1} that
\begin{multline}\label{S1 est2}
S_1(t)\ll T^{1/2}G
\sum_{h=H_0,H_{\infty}}\sum_{m\ll T^{1+\epsilon}\sqrt{t}}\sum_{0<2m-n\ll m\alt_0G^{-1+\epsilon}}
\frac{\mathscr{L}_{n^2-4m^2}(1/2+2it)V(m,\sgh(h)t,T)}{m\left(\alt^2-1+n^2/(4m^2)\right)^{1/4}}\\\times
(4m^2-n^2)^{i(1-\sgh(h))t}
\LE\left(\sgh(h)T,\frac{n}{2m}\right),
\end{multline}
where the function $\LE$ is defined by \eqref{LE def},  $\sgh(h)=-1$ if $h=H_0$ and $\sgh(h)=1$ for $h=H_{\infty}.$ Making  the change of variables
\begin{equation}\label{change of variables2}
2m-n=q,\quad
n+2m=\rr,
\end{equation}
and arguing as  in the case of $S_2(t)$, we conclude that
\begin{equation}\label{S1 to S1(CQL)}
S_1(t)\ll \frac{T^{1/2}G}{\sqrt{\alt_0}}
\sum_{h=H_0,H_{\infty}}
\sum_{C\ll t^{3/2}G^{-1+\epsilon}}\sum_{Q_0/C\ll Q\ll t^{3/2}G^{-1+\epsilon}/C}
\sum_{Q^2C^2G^{1-\epsilon}/\alt_0\ll L\ll QCT^{1+\epsilon}\sqrt{t}}S_1(h,C,Q,L),
\end{equation}
where
\begin{equation}\label{S1(CQL) def}
S_1(h,C,Q,L)=
\sum_{c}U\left(\frac{c}{C}\right)\sum_{q}U\left(\frac{q}{Q}\right)
\sum_{\substack{a\pmod{c}\\(a,c)=1}}\sum_{l}
\frac{\Zag_{-l}(1/2+2it)}{l^{1/2-it}}\phi_1(l)e\left(\frac{-al}{c}\right),
\end{equation}
\begin{equation}\label{S1 phi1 def}
\phi_1(l)=\frac{U(l/L)}{l^{1/2+i\sgh(h)t}}V\left(\frac{l+(qc)^2}{4qc},\sgh(h)t,T\right)
\LE\left(\sgh(h)T,\frac{l-(qc)^2}{l+(qc)^2}\right).
\end{equation}
\section{Integral transforms}\label{sec:integral transforms}
As explained in \cite[Sec. 6]{BFsym2}, it is sufficient to consider only the case $c\equiv0\pmod{4}$. Accordingly, we apply the Voronoi formula \eqref{Voronoi it n>0 c=0 mod 4} to the sum over $l$ in \eqref{S2(CQL) def}.

First of all, it is required to etimate and evaluate the integral transforms of $\phi_2(x)$ arising in  \eqref{Voronoi it n>0 c=0 mod 4}. To this end, we  assume that \eqref{Q0 def} holds and that the conditions \eqref{Q0 conditions GtT2} are satisfied. Under these restrictions, one has
\begin{equation}\label{L conditons}
L\gg\frac{(QC)^2}{\alt_0^2}T^{\epsilon_1},\quad
L\ll\frac{(QC)^2}{\alt_0}T^{1-\epsilon_1}.
\end{equation}
These conditions are satisfied for $L$ in \eqref{S2 def7} since
\begin{equation}\label{L conditons2}
\frac{(QC)^2}{\alt_0^2}T^{\epsilon_1}\ll Q^2C^2G^{1-\epsilon}/\alt_0,\quad
QCT^{1+\epsilon}\sqrt{t}\ll\frac{(QC)^2}{\alt_0}T^{1-\epsilon_1}
\end{equation}
if $\alt_0G\gg T^{\epsilon_0}$, $QC\gg Q_0\gg\alt_0T^{\epsilon_0}\sqrt{t}.$

\begin{lem}\label{lem^phi2+ est}
Suppose that \eqref{Q0 conditions GtT2} and  \eqref{L conditons} hold. Then
\begin{equation}\label{phi2+ est0}
\phi_2^{+}(1/2\pm it)\ll T^{-A}.
\end{equation}
\end{lem}
\begin{proof}
Using  \eqref{S2 phi2 def}, \eqref{LE def} one has
\begin{equation}\label{phi2+ est1}
\phi_2^{+}(1/2\pm it)\sim\int_0^{\infty}U_2(y)e^{2iTF_2(y,\pm)}\frac{dy}{y},
\end{equation}
where
\begin{equation}\label{phi2+ U2def}
U_2(y)=U(y)V\left(\frac{yL-(qc)^2}{4qc},t,T\right)
\exp\left(-G^2\log^2\frac{z+\sqrt{z^2-1+\alt_0^2}}{1+\alt_0}\right),
\end{equation}
\begin{equation}\label{phi2+ F2def}
F_2(y,\pm)=\Le(z)+(\pm1-1)\alt_0\log\sqrt{y},\quad  z=\frac{yL+(qc)^2}{yL-(qc)^2}.
\end{equation}
To estimate \eqref{phi2+ est1} one may apply Lemma \ref{Lemma BKY}.
To this end, it is required to define the parameters  $P,R,V,X,Y$ in \eqref{BKYconditions}.
First, note that
\begin{equation}\label{phi2+ z-1est}
z=1+\frac{2(qc)^2}{yL-(qc)^2}=1+O(\frac{\alt_0G^{\epsilon}}{G}),\quad\hbox{so that}\quad
z^2-1\ll\frac{\alt_0G^{\epsilon}}{G}\ll\alt_0^2,
\end{equation}
and
\begin{equation}\label{phi2+ sqrt deriv z}
\sqrt{z^2-1+\alt_0^2}\sim\alt_0,\quad
\frac{dz}{dy}=\frac{-2L(qc)^2}{(yL-(qc)^2)^2}\sim\frac{(qc)^2}{L}.
\end{equation}
Using \eqref{Le def} one has
\begin{equation}\label{phi2+ h1 deriv}
\frac{d}{dz}\Le(z)=\frac{\alt_0^2-1}{\alt_0z+\sqrt{z^2-1+\alt_0^2}}.
\end{equation}
Finally, since $\frac{(qc)^2}{L\alt_0}\ll\alt_0$, we conclude that
\begin{equation}\label{phi2+ F2 deriv}
\frac{d}{dy}F_2(y,\pm)\gg \frac{(qc)^2}{L\alt_0}+(\pm1-1)\alt_0.
\end{equation}
Therefore, we choose the parameter $R$ in \eqref{BKYconditions} as follows:
\begin{equation}\label{phi2+ Rdef}
R(+)=\frac{T(qc)^2}{L\alt_0},\quad
R(-)=T\alt_0.
\end{equation}
Now let us consider the second partial derivative:
\begin{multline}\label{phi2+ F2 2deriv}
\frac{d^2}{dy^2}F_2(y,\pm)\sim\frac{d^2}{dz^2}\Le(z)\frac{(qc)^4}{L^2}+
\frac{d}{dz}\Le(z)\frac{(qc)^2}{L}+(\pm1-1)\alt_0\sim
\frac{(qc)^4}{L^2\alt_0^3}+\\+
\frac{(qc)^2}{L\alt_0}+(\pm1-1)\alt_0\sim
\frac{(qc)^2}{L\alt_0}+(\pm1-1)\alt_0.
\end{multline}
Arguing in the same way to estimate $F^{(j)}_2(y,\pm)$, we infer
\begin{equation}\label{phi2+ YQdef}
Y(\pm)=R(\pm),\quad
P(\pm)=1.
\end{equation}
Estimating \eqref{phi2+ U2def} trivially one has  $X=T^{\epsilon}$. Finally,
\begin{multline}\label{phi2+ U2deriv}
\frac{d}{dy}\exp\left(-G^2\log^2\frac{z+\sqrt{z^2-1+\alt_0^2}}{1+\alt_0}\right)\ll
G^2\log\left(\frac{z+\sqrt{z^2-1+\alt_0^2}}{1+\alt_0}\right)\frac{1}{\sqrt{z^2-1+\alt_0^2}}\frac{dz}{dy}\ll\\\ll
G^{1+\epsilon}\frac{(qc)^2}{L\alt_0}\ll T^{\epsilon}.
\end{multline}
since $L\gg Q^2C^2G^{1-\epsilon}/\alt_0$. Hence  $V=T^{-\epsilon}$ and applying Lemma \ref{Lemma BKY} one has
\begin{equation}\label{phi2+ est2}
\phi_2^{+}(1/2\pm it)\ll
\frac{T^{\epsilon}}{\left(R(\pm)T^{-\epsilon}\right)^{A}}\ll T^{-A}
\end{equation}
provided that $R(\pm)\gg T^{\epsilon_0}.$ These inequalities follow from \eqref{phi2+ Rdef} together with \eqref{Q0 conditions GtT2} and \eqref{L conditons}.
\end{proof}

Now we are going to investigate the integral transforms of  $\phi_2(x)$ arising on the right-hand side of \eqref{Voronoi it n>0 c=0 mod 4}.

\begin{lem}\label{lem^phi2+hat est}
Suppose that \eqref{Q0 conditions GtT2} and  \eqref{L conditons} are satisfied. Then
\begin{equation}\label{phi2+hat  est1}
\widehat{\phi_2}^{+}\left(\frac{4\pi^2n}{c^2}\right)\ll \frac{1}{(nT)^{A}}\quad\hbox{for}\quad
n\gg\frac{C^2t^2}{L}T^{\epsilon},
\end{equation}
\begin{equation}\label{phi2+hat  est2}
\widehat{\phi_2}^{+}\left(\frac{4\pi^2n}{c^2}\right)\ll \frac{1}{T^{A}}\quad\hbox{for}\quad
0<n\ll\frac{C^2t^2}{L}T^{\epsilon}.
\end{equation}
\end{lem}
\begin{proof}
Let $y=\frac{2\pi\sqrt{n}}{c}$. It follows from \eqref{phi+ transform def}, \eqref{Phi++ itdef} and \eqref{S2 phi2 def} that
\begin{equation}\label{phi2+hat  int1}
\widehat{\phi_2}^{+}\left(y^2\right)=
\frac{y}{\sqrt{2}}\int_{0}^{\infty}\left(F_{2it}(2y\sqrt{x})-G_{2it}(2y\sqrt{x}\right)
\frac{U(x/L)}{x^{1+it}}V\left(\frac{x-(qc)^2}{4qc},t,T\right)
\LE\left(\frac{x+(qc)^2}{x-(qc)^2}\right)dx.
\end{equation}
Due to the properties of functions  $F_{2it}(z)$ and $G_{2it}(z)$,  the two integrals in the formula above can be treated in the same way.
To prove \eqref{phi2+hat  est1}, we first perform the change of  variables $x=\frac{4t^2}{y^2}z^2$, then use the relation
\eqref{FG Dun difeq}, and finally, integrate by parts twice, obtaining
\begin{multline}\label{phi2+hat  int2}
\widehat{\phi_2}^{+}\left(y^2\right)\ll
y\int_{0}^{\infty}\frac{d^2}{dz^2}\Biggl(U\left(\frac{4t^2z^2}{y^2L}\right)
\frac{z^{1/2-2it}}{1+16t^2(1+z^2)}
V\left(\frac{4t^2z^2/y^2-(qc)^2}{4qc},t,T\right)\\\times
\LE\left(\frac{z^2+(yqc/2t)^2}{z^2-(yqc/2t)^2}\right)\Biggr)\sqrt{z}F_{2it}(2tz)dz.
\end{multline}
Since $n\gg\frac{C^2t^2}{L}T^{\epsilon}$, we conclude that on the interval of integration  $z\sim\frac{y\sqrt{L}}{t}\gg (nT)^{\epsilon}$. Using \eqref{LE def},
\eqref{phi2+ h1 deriv} and \eqref{phi2+ U2deriv} one has
\begin{multline}\label{phi2+hat  int3}
\frac{d}{dz}\LE\left(\frac{z^2+(yqc/2t)^2}{z^2-(yqc/2t)^2}\right)\ll\LE'\left(\frac{z^2+(yqc/2t)^2}{z^2-(yqc/2t)^2}\right)\frac{(yqc/t)^2}{z^3}\ll\\\lll
\left(\frac{T}{\alt_0}+G^{1+\epsilon}\frac{(qc)^2}{L\alt_0}\right)\frac{q^2c^2}{Lz}
\ll\frac{Tq^2c^2}{Lz\alt_0}\ll\frac{T^{1+\epsilon}}{zG}\ll\frac{t}{z}.
\end{multline}
Accordingly, after integration by parts we obtain an integral similar to the original one but multiplied by $$\frac{t^2}{1+t^2(1+z^2)}\ll\frac{1}{z^2}\ll\frac{1}{(nT)^{\epsilon}}.$$ Repeated application of such integration by parts leads to \eqref{phi2+hat  est1}.

In the remaining case: $0<n\ll\frac{C^2t^2}{L}T^{\epsilon}$, we first make the change of variables $x=c^2z$ in \eqref{phi2+hat  int1}:
\begin{equation}\label{phi2+hat  int4}
\widehat{\phi_2}^{+}\left(y^2\right)\ll
y\int_{0}^{\infty}\left(F_{2it}(4\pi\sqrt{nz})-G_{2it}(4\pi\sqrt{nz}\right)
\frac{U(c^2z/L)}{z^{1+it}}V\left(c\frac{z-q^2}{4q},t,T\right)
\LE\left(\frac{z+q^2}{z-q^2}\right)dz.
\end{equation}
Now we are going to substitute the asymptotic formula \eqref{GF2tz asympt}. As usual, it is enough to consider only the contribution of the main term
\begin{equation}\label{phi2+hat  int5}
\widehat{\phi_2}^{+}\left(y^2\right)\ll
\frac{\sqrt{n}}{c\sqrt{t}}\sum_{\pm}\int_{0}^{\infty}U_3(z)e^{2iTF_3(z,\pm)}dz,
\end{equation}
where
\begin{equation}\label{phi2+hat U3def}
U_3(z)=
\frac{U(c^2z/L)}{z\left(1+\frac{4\pi^2nz}{t^2}\right)^{1/4}}V\left(c\frac{z-q^2}{4q},t,T\right)
\exp\left(-G^2\log^2\frac{x(z)+\sqrt{x^2(z)-1+\alt_0^2}}{1+\alt_0}\right),
\end{equation}
\begin{equation}\label{phi2+hat F3def}
F_3(z,\pm)=\Le(x(z))\pm\alt_0\xi\left(\frac{2\pi\sqrt{nz}}{t}\right)-\alt_0\log\sqrt{z},\quad  x(z)=\frac{z+q^2}{z-q^2}.
\end{equation}
Now we argue in the same way as in Lemma \ref{lem^phi2+ est}.
To estimate \eqref{phi2+hat  int5}, one can apply Lemma \ref{Lemma BKY}.
To this end, we need to define parameters  $P,R,V,X,Y$ in \eqref{BKYconditions}.
Using \eqref{GF2tz xi def} and \eqref{phi2+ h1 deriv}, we obtain
\begin{equation}\label{phi2+hat F3deriv}
\frac{d}{dz}F_3(z,\pm)=\frac{1-\alt_0^2}{\alt_0x(z)+\sqrt{x(z)^2-1+\alt_0^2}}\frac{2q^2}{(z-q^2)^2}
\pm\frac{\alt_0}{2z}\sqrt{1+\frac{4\pi^2nz}{t^2}}-\frac{\alt_0}{2z}.
\end{equation}
Since $z\sim L/c^2$, it follows from \eqref{L conditons} that $\frac{q^2}{\alt_0z^2}\ll\frac{\alt_0}{z}T^{-\epsilon}$. Therefore,
\begin{equation}\label{phi2+hat F3deriv2}
\left|\frac{d}{dz}F_3(z,-)\right|\gg\frac{\alt_0}{2z},\quad
\frac{d}{dz}F_3(z,+)\gg\frac{q^2}{\alt_0z^2}+\frac{\alt_0n}{t^2\sqrt{1+\frac{4\pi^2nz}{t^2}}}\gg
\frac{q^2}{\alt_0z^2}+\frac{\alt_0n}{t^2T^{\epsilon}}.
\end{equation}
Note that the derivative is smaller in the plus case, and therefore, it is sufficient to study further only this case. Accordingly, we choose
\begin{equation}\label{phi2+hat F3 R def}
R:=T\frac{q^2c^4}{\alt_0L^2}+\frac{\alt_0n}{t^2}T^{1-\epsilon}.
\end{equation}
Evaluating the second derivative of $F_3(z,+)$, one has
\begin{equation}\label{phi2+hat F3 2deriv}
\frac{d^2}{dz^2}F_3(z,+)\ll\frac{1}{\alt_0^3}\frac{q^4}{z^4}+\frac{q^2}{\alt_0z^3}+\frac{\alt_0n^2}{t^4}\ll
\frac{q^2}{\alt_0z^3}+\frac{\alt_0n^2}{t^4}\ll
\frac{q^2}{\alt_0z^3}+\frac{\alt_0n}{t^2}\frac{T^{\epsilon}}{z}.
\end{equation}
Therefore,
\begin{equation}\label{phi2+hat F3 QY def}
P=\frac{L}{c^2T^{\epsilon}},\quad Y=PR.
\end{equation}
Calculating the derivatives of $U_3(z)$, using the estimate
\begin{equation}\label{phi2+hat U3deriv}
\frac{d}{dz}\exp\left(-G^2\log^2\frac{x(z)+\sqrt{x^2(z)-1+\alt_0^2}}{1+\alt_0}\right)\ll\frac{G^{1+\epsilon}}{\alt_0}\frac{dx}{dz}\ll
\frac{G^{1+\epsilon}q^2}{\alt_0z^2}\ll\frac{G^{1+\epsilon}(qc)^2}{\alt_0Lz}\ll\frac{T^{\epsilon}}{z},
\end{equation}
which holds since  $L\gg(QC)^2G^{1-\epsilon}/\alt_0$, we finally conclude that $V=\frac{L}{c^2T^{\epsilon}}.$
Applying Lemma \ref{Lemma BKY} to estimate \eqref{phi2+hat  int5} we infer
\begin{equation}\label{phi2+hat  int6}
\widehat{\phi_2}^{+}\left(y^2\right)\ll
\frac{\sqrt{n}}{c\sqrt{t}}\left(T\frac{q^2c^2}{\alt_0LT^{\epsilon}}+\frac{\alt_0Ln}{c^2t^2}T^{1-\epsilon}\right)^{-A}.
\end{equation}
Suppose that $n\gg\frac{(qc^2t)^2}{(\alt_0L)^2}$.  Under this assumption, using \eqref{L conditons} one has
\begin{equation}\label{phi2+hat  int7.1}
\frac{\alt_0Ln}{c^2t^2}T^{1-\epsilon}\gg T^{1-\epsilon}\frac{(qc)^2}{\alt_0L}\gg T^{\epsilon},
\end{equation}
and therefore,
\begin{equation}\label{phi2+hat  int7}
\widehat{\phi_2}^{+}\left(y^2\right)\ll
\frac{\sqrt{n}}{c\sqrt{t}}\left(\frac{\alt_0Ln}{c^2t^2}T^{1-\epsilon}\right)^{-A}\ll T^{-A}.
\end{equation}
Suppose that $0<n\ll\frac{(qc^2t)^2}{(\alt_0L)^2}$.  Under this assumption,  according to  \eqref{L conditons} one has the estimate
\begin{equation}\label{phi2+hat  int8}
\widehat{\phi_2}^{+}\left(y^2\right)\ll
\frac{\sqrt{n}}{c\sqrt{t}}\left(T\frac{q^2c^2}{\alt_0LT^{\epsilon}}\right)^{-A}\ll T^{-A}.
\end{equation}
This completes the proof.
\end{proof}

\begin{lem}\label{lem:phi2+hat<0 est}
Suppose that \eqref{Q0 conditions GtT2} and  \eqref{L conditons} hold. Then
\begin{equation}\label{phi2+hat<0  est1}
\widehat{\phi_2}^{+}\left(-\frac{4\pi^2n}{c^2}\right)\ll \frac{1}{(nT)^{A}},\quad\hbox{for}\quad
0<n\ll\frac{Q^2C^4t^2}{\alt_0^2L^2},\quad
n\gg\frac{Q^2C^4t^2}{\alt_0^2L^2},
\end{equation}
\begin{equation}\label{phi2+hat<0  sp1}
\widehat{\phi_2}^{+}\left(-\frac{4\pi^2n}{c^2}\right)\sim
\frac{U_0(z_0)e^{2iTF_0(z_0)}}{c^{1+2it}\sqrt{z_0}},
\quad\hbox{for}\quad n\sim\frac{Q^2C^4t^2}{\alt_0^2L^2},
\end{equation}
where
\begin{equation}\label{U0def}
U_0(z)=U\left(\frac{c^2z}{L}\right)V\left(c\frac{z-q^2}{4q},t,T\right)
\exp\left(-G^2\log^2\frac{x(z)+\sqrt{x^2(z)-1+\alt_0^2}}{1+\alt_0}\right),
\end{equation}
\begin{equation}\label{F0def}
F_0(z)=\Le\left(\frac{z+q^2}{z-q^2}\right)-\alt_0\eta\left(\frac{2\pi\sqrt{nz}}{t}\right)-\alt_0\log\sqrt{z},
\end{equation}
$\eta(v)$ is defined in \eqref{K2tz eta def} and $z_0$ is a solution of
\begin{equation}\label{F5-z0}
\frac{1-\alt_0^2}{\alt_0x(z)+\sqrt{x(z)^2-1+\alt_0^2}}\frac{2q^2}{(z-q^2)^2}=\frac{2\pi^2n\alt_0/t^2}{1+\sqrt{1-\frac{4\pi^2nz}{t^2}}},
\quad  x(z)=\frac{z+q^2}{z-q^2}.
\end{equation}
Moreover, one has $z_0\approx\frac{qt}{\alt_0\sqrt{n}}$ and $z_0$ is independent of $c$.
\end{lem}
\begin{proof}
Let $y=2\pi\sqrt{n}/c.$ It follows from \eqref{phi+ transform def}, \eqref{Phi-+ itdef} and \eqref{S2 phi2 def} that
\begin{multline}\label{phi2+hat<0  int1}
\widehat{\phi_2}^{+}\left(-y^2\right)=
\frac{2y}{\pi}\sin(\pi/4+\pi it)\\ \times \int_{0}^{\infty}K_{2it}(2y\sqrt{x})
\frac{U(x/L)}{x^{1+it}}V\left(\frac{x-(qc)^2}{4qc},t,T\right)
\LE\left(\frac{x+(qc)^2}{x-(qc)^2}\right)dx.
\end{multline}
For $n\gg\frac{C^2t^2}{L}T^{\epsilon}$  the integral \eqref{phi2+hat<0  int1} is bounded by $(nT)^{-A}$. The proof of this fact is similar to the one of \eqref{phi2+hat  est1},
since the differential equations \eqref{K Dun difeq} and \eqref{FG Dun difeq} differ slightly from each other in the range $z\gg T^{\epsilon}.$
From now on, we may assume that $0<n\ll\frac{C^2t^2}{L}T^{\epsilon}$. Making the change of variables $x=c^2z$ in \eqref{phi2+hat<0  int1}, one has
\begin{equation}\label{phi2+hat<0 int2}
\widehat{\phi_2}^{+}\left(-y^2\right)=
\frac{2y}{\pi c^{2it}}\sin(\pi/4+\pi it)\int_{0}^{\infty}K_{2it}(4\pi\sqrt{nz})
\frac{U(c^2z/L)}{z^{1+it}}V\left(c\frac{z-q^2}{4q},t,T\right)
\LE\left(\frac{z+q^2}{z-q^2}\right)dz.
\end{equation}
The behaviour of $K_{2it}(x)$ is different for small $x$, large $x$ and $x\sim1$. Accordingly, we use the following smooth partition of unity:
\begin{equation}\label{partition of 1}
\chi_0\left(\frac{2\pi\sqrt{nz}}{t}\right)+\chi_1\left(\frac{2\pi\sqrt{nz}}{t}\right)+\chi_{\infty}\left(\frac{2\pi\sqrt{nz}}{t}\right)=1,
\end{equation}
where
\begin{equation}\label{partition of 1 0}
\chi_0(x)=1\,\hbox{ for }\,0<x<1-2t^{\epsilon-2/3},\,
\chi_0(x)=0\,\hbox{ for }\,x>1-t^{\epsilon-2/3},
\end{equation}
\begin{equation}\label{partition of 1 infty}
\chi_{\infty}(x)=1\,\hbox{ for }\,x>1+2t^{\epsilon-2/3},\,
\chi_{\infty}(x)=0\,\hbox{ for }\,x<1+t^{\epsilon-2/3},
\end{equation}
\begin{equation}\label{partition of 1 1}
\chi_{1}(x)=1\,\hbox{ for }\,|x-1|<t^{\epsilon-2/3},\,
\chi_{1}(x)=0\,\hbox{ for }\,|x-1|>2t^{\epsilon-2/3}.
\end{equation}
The part of the integral \eqref{phi2+hat<0 int2} with $\chi_{\infty}\left(\frac{2\pi\sqrt{nz}}{t}\right)$ is negligible in view of \eqref{K2tz z>1 est}. Let us consider the part with  $\chi_{1}$ function. One has
\begin{equation}\label{phi2+hat<0 int3}
\widehat{\phi_2}^{+}\left(-y^2\right)\ll
\int_{0}^{\infty}U_4(z)e^{2iTF_4(z)}dz,
\end{equation}
where
\begin{multline}\label{phi2+hat U4def}
U_4(z)=e^{\pi t}K_{2it}(4\pi\sqrt{nz})\chi_{1}\left(\frac{2\pi\sqrt{nz}}{t}\right)
\frac{U(c^2z/L)}{z}V\left(c\frac{z-q^2}{4q},t,T\right)\\\times
\exp\left(-G^2\log^2\frac{x(z)+\sqrt{x^2(z)-1+\alt_0^2}}{1+\alt_0}\right),
\end{multline}
\begin{equation}\label{phi2+hat F4def}
F_4(z)=\Le(x(z))-\alt_0\log\sqrt{z}.
\end{equation}
Once again we intend to apply Lemma \ref{Lemma BKY}. With this goal, we compute derivatives of $F_4(z)$ (see \eqref{phi2+hat F3deriv} for similar computations) and choose parameters in \eqref{BKYconditions} as folows:
\begin{equation}\label{F4 QRY def}
R=\frac{tc^2}{L},\,
P=\frac{c^2}{L},\, Y=t.
\end{equation}
According to \eqref{K2tz z=1 est} one has
\begin{equation}
e^{\pi t}\frac{\partial^j}{\partial z^j}K_{2it}(4\pi\sqrt{nz})\ll\frac{(t^{2/3}/z)^j}{t^{1/3}}.
\end{equation}
Therefore, evaluating the derivatives of $U_4(z)$, we set $$V=\frac{L}{c^2t^{2/3}}$$ in  Lemma \ref{Lemma BKY}. Finally, we  show that
\begin{equation}\label{phi2+hat<0 int4}
\widehat{\phi_2}^{+}\left(-y^2\right)\ll
\frac{T^{\epsilon}}{t^{1/3}}\left(\frac{1}{(RP)^A}+\frac{1}{(RV)^A}\right)\ll t^{-A/3}.
\end{equation}

It is left to consider the part of the integral \eqref{phi2+hat<0 int2} with $\chi_{0}\left(\frac{2\pi\sqrt{nz}}{t}\right)$.
As usual, it is enough to consider only the contribution of the main term in the asymptotic expansion \eqref{K2tz asympt}:
\begin{equation}\label{phi2+hat<0 int5}
\widehat{\phi_2}^{+}\left(-y^2\right)\sim\frac{\sqrt{n}}{c^{1+2it}\sqrt{t}}
\int_{0}^{\infty}U_5(z)e^{2iTF_5(z,\pm)}dz,
\end{equation}
where
\begin{equation}\label{phi2+hat U5def}
U_5(z)=\chi_{0}\left(\frac{2\pi\sqrt{nz}}{t}\right)
\frac{U(c^2z/L)}{z\left(1-\frac{4\pi^2nz}{t^2}\right)^{1/4}}V\left(c\frac{z-q^2}{4q},t,T\right)
\exp\left(-G^2\log^2\frac{x+\sqrt{x^2-1+\alt_0^2}}{1+\alt_0}\right),
\end{equation}
\begin{equation}\label{phi2+hat F5def}
F_5(z,\pm)=\Le(x(z))\pm\alt_0\eta\left(\frac{2\pi\sqrt{nz}}{t}\right)-\alt_0\log\sqrt{z},\quad  x(z)=\frac{z+q^2}{z-q^2}.
\end{equation}
Note that $U_5(z)\neq0$ only if $$n<\frac{c^2t^2}{2\pi^2L}(1-t^{-2/3+\epsilon})^2.$$
Using \eqref{K2tz eta def} and \eqref{phi2+ h1 deriv}, we prove that
\begin{equation}\label{phi2+hat F5deriv}
\frac{d}{dz}F_5(z,\pm)=\frac{1-\alt_0^2}{\alt_0x(z)+\sqrt{x(z)^2-1+\alt_0^2}}\frac{2q^2}{(z-q^2)^2}
\mp\frac{\alt_0}{2z}\sqrt{1-\frac{4\pi^2nz}{t^2}}-\frac{\alt_0}{2z}.
\end{equation}
One has $$\frac{d}{dz}F_5(z,+)\gg\frac{\alt_0}{2z}.$$ Consequently, this case is easier than the second one:
\begin{equation}\label{phi2+hat F5deriv-}
\frac{d}{dz}F_5(z,-)=\frac{1-\alt_0^2}{\alt_0x(z)+\sqrt{x(z)^2-1+\alt_0^2}}\frac{2q^2}{(z-q^2)^2}
-\frac{2\pi^2n\alt_0/t^2}{1+\sqrt{1-\frac{4\pi^2nz}{t^2}}}.
\end{equation}
The saddle point $z_0$ is a solution of the equation \eqref{F5-z0}.
It seems difficult to write down an explicit formula for $z_0$. Nevertheless, we know that the left-hand side of \eqref{F5-z0} is approximately $q^2\alt_0^{-1}z_0^{-2}$, and therefore,
$z_0\approx\frac{qt}{\alt_0\sqrt{n}}$. This point belongs to the interval of integration only if
\begin{equation}\label{N def}
n\sim\frac{q^2c^4t^2}{\alt_0^2L^2}=:N.
\end{equation}
Hence it is required to consider three cases: $n\ll N,$ $n\sim N$ and $n\gg N$.


First, suppose that $n\ll N.$ In order to apply Lemma \ref{Lemma BKY}, it is require to choose parameters in \eqref{BKYconditions}.
According to \eqref{L conditons} one has $$\frac{q^2c^4t^2}{\alt_0^2L^2}\ll\frac{c^2t^2}{2L}T^{-\epsilon}.$$ As a result, for $n<cN$ the following estimate holds:  $$\frac{4\pi^2nz}{t^2}\ll T^{-\epsilon}.$$
Therefore,  $\chi_{0}\left(\frac{2\pi\sqrt{nz}}{t}\right)=1$. Furthermore, it is possible to expand $\left(1-\frac{4\pi^2nz}{t^2}\right)^{1/4}$ in  the Taylor series so that $U_5(z)$ simplifies to
\begin{equation}\label{phi2+hat U5def2}
U_5(z)=
\frac{U(c^2z/L)}{z}V\left(c\frac{z-q^2}{4q},t,T\right)
\exp\left(-G^2\log^2\frac{x+\sqrt{x^2-1+\alt_0^2}}{1+\alt_0}\right).
\end{equation}
Consequently, (as in the case of the $U_3(z)$ function in Lemma \ref{lem^phi2+hat est})  we choose $V=L/(c^2T^{\epsilon})$ in  \eqref{BKYconditions}.
Moreover, since $\frac{4\pi^2nz}{t^2}\ll T^{-\epsilon}$ it follows from \eqref{phi2+hat F3deriv}, \eqref{phi2+hat F5deriv-} that the parameters for $F_5$ can be identified  in the same manner as for $F_3$ (see \eqref{phi2+hat F3 R def} and \eqref{phi2+hat F3 QY def}):
\begin{equation}\label{phi2+hat F5 RPY def}
R=T\frac{q^2c^4}{\alt_0L^2}+\frac{\alt_0n}{t^2}T\gg T\frac{q^2c^4}{\alt_0L^2}, \quad P=\frac{L}{c^2},\quad Y=PR.
\end{equation}
Finally, using \eqref{I BKY est} and \eqref{L conditons} we infer that
\begin{equation}\label{phi2+hat<0 int6}
\int_{0}^{\infty}U_5(z)e^{2iTF_5(z,\pm)}dz\ll T^{-A}.
\end{equation}

Second, suppose that $N\ll n<\delta\frac{c^2t^2}{L}.$  The constant $\delta$ is chosen in such a way that
\begin{equation}
\chi_{0}\left(\frac{2\pi\sqrt{nz}}{t}\right)=1\quad\hbox{and}\quad\left(1-\frac{4\pi^2nz}{t^2}\right)^{1/4}\gg1.
\end{equation}
This guarantees that as in the previous case one can take $V=L/(c^2T^{\epsilon})$ in  \eqref{BKYconditions}.  Furthermore,
\begin{equation}\label{phi2+hat F5deriv-2}
T|\frac{d}{dz}F_5(z,-)|\gg R:=T\frac{n\alt_0}{t^2}=\frac{n}{t},
\end{equation}
\begin{equation}\label{phi2+hat F5- 2deriv}
\frac{d^2}{dz^2}F_5(z,-)\ll \frac{q^2}{\alt_0z^3}+
\frac{n\alt_0}{t^2}\frac{n}{t^2}.
\end{equation}
Therefore, since $n/t^2\ll c^2/L$ one has $P=\frac{L}{c^2}$, $Y=PR.$ It follows from  \eqref{L conditons}, \eqref{N def} that
\begin{equation}
RP>RV>\frac{nL}{tc^2T^{\epsilon}}\gg\frac{NL}{tc^2T^{\epsilon}}\gg\frac{q^2c^2t}{\alt_0^2LT^{\epsilon}}\gg
\frac{Q^2C^2T^{1-\epsilon}}{\alt_0L}\gg  T^{\epsilon}.
\end{equation}
As a result,
\begin{equation}\label{phi2+hat<0 int7}
\int_{0}^{\infty}U_5(z)e^{2iTF_5(z,\pm)}dz\ll \frac{1}{(RV)^{A}}+\frac{1}{(RP)^{A}}\ll T^{-A}.
\end{equation}

Third, suppose that $$\delta\frac{c^2t^2}{L}<n<\frac{c^2t^2}{2\pi^2L}(1-t^{-2/3+\epsilon})^2.$$ In that case, the parameter $R$ is the same as in the previous case, namely $$R=T\frac{n\alt_0}{t^2}=\frac{n}{t},$$ but $\chi_{0}(\cdot)$ may not be equal to one, and therefore,
\begin{equation}\label{U5 deriv}
\frac{d}{dz}
\frac{\chi_{0}\left(\frac{2\pi\sqrt{nz}}{t}\right)}{\left(1-\frac{4\pi^2nz}{t^2}\right)^{1/4}}\ll
\frac{\sqrt{nz}}{tz}\frac{t^{2/3-\epsilon}}{\left(1-\frac{4\pi^2nz}{t^2}\right)^{1/4}}+\left(1-\frac{4\pi^2nz}{t^2}\right)^{-5/4}\frac{n}{t^2}\ll
\left(1-\frac{4\pi^2nz}{t^2}\right)^{-1/4}\frac{t^{2/3-\epsilon}}{z}.
\end{equation}
Accordingly, we set $$V=\frac{L}{c^2t^{2/3-\epsilon}}, X=t^{1/6-\epsilon}$$ in  \eqref{BKYconditions}. It is left to determine the parameters $P$ and $Y$. Note that
\begin{equation}\label{phi2+hat F5- 2deriv2}
\frac{d^2}{dz^2}F_5(z,-)\ll \frac{q^2}{\alt_0z^3}+
\frac{n\alt_0}{t^2}\frac{n}{t^2\sqrt{1-\frac{4\pi^2nz}{t^2}}}.
\end{equation}
Moreover, while computing higher derivatives, a multiple $\frac{n}{t^2\left(1-\frac{4\pi^2nz}{t^2}\right)}$ appears. Hence $$P=\frac{L}{t^{2/3-\epsilon}c^2}, \quad Y=PR.$$ Furthermore, since
\begin{equation}
RP=RV>\frac{n}{t}\frac{L}{t^{2/3-\epsilon}c^2}\gg
t^{1/3+\epsilon},
\end{equation}
we prove that
\begin{equation}\label{phi2+hat<0 int8}
\int_{0}^{\infty}U_5(z)e^{2iTF_5(z,\pm)}dz\ll \frac{t^{1/6-\epsilon}}{(RV)^{A}}+\frac{t^{1/6-\epsilon}}{(RP)^{A}}\ll T^{-A}.
\end{equation}

Finally, it is  left to consider the case $n\sim\frac{q^2c^4t^2}{\alt_0^2L^2}=N$ (that is to prove \eqref{phi2+hat<0  sp1}). In this case,  the function $U_5(z)$ may be simplified to \eqref{phi2+hat U5def2} (and thus $V=\frac{L}{c^2T^{\epsilon}}$).  Evaluating the derivative of \eqref{phi2+hat F5deriv-}, we show that
\begin{multline}\label{phi2+hat F5- 2deriv}
\frac{d^2}{dz^2}F_5(z,-)=
\frac{1-\alt_0^2}{\left(\alt_0x(z)+\sqrt{x(z)^2-1+\alt_0^2}\right)^2}\left(\alt_0+\frac{x(z)}{\sqrt{x(z)^2-1+\alt_0^2}}\right)\frac{4q^4}{(z-q^2)^4}-\\
-\frac{1-\alt_0^2}{\alt_0x(z)+\sqrt{x(z)^2-1+\alt_0^2}}\frac{4q^2}{(z-q^2)^3}
-\frac{\alt_0(2\pi^2n/t^2)^2}{\left(1+\sqrt{1-\frac{4\pi^2nz}{t^2}}\right)^2\sqrt{1-\frac{4\pi^2nz}{t^2}}}.
\end{multline}
Up to some constants,
\begin{equation}\label{phi2+hat F5- 2deriv2}
\frac{d^2}{dz^2}F_5(z,-)\sim
\frac{q^4}{\alt_0^3z^4}-\frac{q^2}{\alt_0z^3}-\frac{\alt_0n^2}{t^4}.
\end{equation}
Since $n\sim N$, one has $\frac{\alt_0n^2}{t^4}\sim\frac{q^4}{\alt_0^3z^4}\ll\frac{q^2}{\alt_0z^3}T^{-\epsilon}$. Therefore,
\begin{equation}\label{phi2+hat F5- 2deriv3}
\frac{d^2}{dz^2}F_5(z,-)\sim-\frac{q^2}{\alt_0z^3} \quad\hbox{and}\quad
|\frac{d^j}{dz^j}F_5(z,-)|\ll\frac{q^2}{\alt_0z^{1+j}}.
\end{equation}
Using \eqref{F5-z0}, one can write down an asymptotic formula for $F_5''(z_0,-)$, but  the main term  in this formula will still depend on $z_0$:
\begin{multline}\label{phi2+hat F5- 2derivasympt}
\frac{d^2}{dz^2}F_5(z_0,-)=
-\frac{1-\alt_0^2}{\alt_0x(z_0)+\sqrt{x(z_0)^2-1+\alt_0^2}}\frac{4q^2}{(z-q^2)^3}\left(1+O(T^{-\epsilon})\right)=\\=
\frac{-4\pi^2n\alt_0/t^2}{\left(1+\sqrt{1-\frac{4\pi^2nz}{t^2}}\right)(z_0-q^2)}\left(1+O(T^{-\epsilon})\right)=
\frac{-2\pi^2n\alt_0}{z_0t^2}\left(1+O(T^{-\epsilon})\right).
\end{multline}

In order to obtain an asymptotic expansion for \eqref{phi2+hat<0 int5}, we apply the saddle point method in a form given in \cite[Proposition 8.2]{BKY}.
We use \cite[Proposition 8.2]{BKY} with the parameters:
\begin{equation}\label{BKY8.2 param}
X=\frac{c^2T^{\epsilon}}{L},\quad
V=\frac{L}{c^2T^{\epsilon}},\quad
V_1=\frac{L}{c^2},\quad
Y=\frac{Tq^2c^2}{\alt_0L},\quad
Q=\frac{L}{c^2}.
\end{equation}
The conditions in \cite[Proposition 8.2]{BKY}, namely $V_1>V>QT^{\epsilon}/\sqrt{Y}$ and $Y>T^{\epsilon}$, are satisfied in view of
\eqref{L conditons}. Finally,  using \eqref{phi2+hat F5- 2derivasympt}, one has
\begin{equation}\label{phi2+hat<0 int sp}
\widehat{\phi_2}^{+}\left(-y^2\right)\sim\frac{\sqrt{n}}{c^{1+2it}\sqrt{t}}
\frac{U_5(z_0)e^{2iTF_5(z_0,-)}}{\sqrt{T|F_5''(z_0,-)|}}\sim
\frac{\sqrt{z_0}}{c^{1+2it}}
U_5(z_0)e^{2iTF_5(z_0,-)}.
\end{equation}
\end{proof}

Our next step is to evaluate and estimate the integral transforms of $\phi_1(x)$ (see \eqref{S1 phi1 def}) arising in  \eqref{Voronoi it n<0 c=0 mod 4}. To this end, one can use techniques similar to those applied for $\phi_2(x).$

\begin{lem}\label{lem^phi1- est}
Suppose that \eqref{Q0 conditions GtT2} and  \eqref{L conditons} hold. Then
\begin{equation}\label{phi1- est0}
\phi_1^{-}(1/2\pm it)\ll T^{-A}.
\end{equation}
\end{lem}
\begin{proof}
Using  \eqref{S1 phi1 def}, \eqref{LE def} one has
\begin{equation}\label{phi2+ est1}
\phi_1^{-}(1/2\pm it)\sim\int_0^{\infty}U_6(y)e^{2i\sgh(h)TF_6(y,\pm)}\frac{dy}{y},
\end{equation}
where
\begin{equation}\label{phi1- U6def}
U_6(y)=U(y)V\left(\frac{yL+(qc)^2}{4qc},\sgh(h)t,T\right)
\exp\left(-G^2\log^2\frac{z+\sqrt{z^2-1+\alt_0^2}}{1+\alt_0}\right),
\end{equation}
\begin{equation}\label{phi1- F6def}
F_6(y,\pm)=\Le(z)+(\pm\sgh(h)-1)\alt_0\log\sqrt{y},\quad  z=\frac{yL-(qc)^2}{yL+(qc)^2}.
\end{equation}
Now arguing as in Lemma \ref{lem^phi2+ est} we prove \eqref{phi1- est0}.
\end{proof}


\begin{lem}\label{lem:phi1-hat<0 est}
Suppose that \eqref{Q0 conditions GtT2} and  \eqref{L conditons} hold. Then
\begin{equation}\label{phi1-hat<0  est1}
\widehat{\phi_1}^{-}\left(-\frac{4\pi^2n}{c^2}\right)\ll \frac{1}{(nT)^{A}},\quad\hbox{for}\quad
0<n\ll\frac{Q^2C^4t^2}{\alt_0^2L^2},\quad
n\gg\frac{Q^2C^4t^2}{\alt_0^2L^2},
\end{equation}
\begin{equation}\label{phi1-hat<0  sp1}
\widehat{\phi_1}^{-}\left(-\frac{4\pi^2n}{c^2}\right)\sim
\frac{U_1(z_1)e^{2i\sgh(h)TF_1(z_1)}}{c^{1+2i\sgh(h)t}\sqrt{z_1}},
\quad\hbox{for}\quad n\sim\frac{Q^2C^4t^2}{\alt_0^2L^2},
\end{equation}
where
\begin{equation}\label{U1def}
U_1(z)=U\left(\frac{c^2z}{L}\right)V\left(c\frac{z+q^2}{4q},\sgh(h)t,T\right)
\exp\left(-G^2\log^2\frac{x(z)+\sqrt{x^2(z)-1+\alt_0^2}}{1+\alt_0}\right),
\end{equation}
\begin{equation}\label{F1def}
F_1(z)=\Le\left(\frac{z-q^2}{z+q^2}\right)+\alt_0\xi\left(\frac{2\pi\sqrt{nz}}{t}\right)-\alt_0\log\sqrt{z},
\end{equation}
$\xi(v)$ is defined in \eqref{GF2tz xi def} and $z_1$ is a solution of the equation:
\begin{equation}\label{z1 def}
\frac{1-\alt_0^2}{\alt_0x(z)+\sqrt{x(z)^2-1+\alt_0^2}}\frac{2q^2}{(z+q^2)^2}=\frac{2\pi^2n\alt_0/t^2}{1+\sqrt{1+\frac{4\pi^2nz}{t^2}}},
\quad  x(z)=\frac{z-q^2}{z+q^2}.
\end{equation}
Moreover, one has $z_1\approx\frac{qt}{\alt_0\sqrt{n}}$ and $z_1$ is independent of $c$.
\end{lem}
\begin{proof}
Let $y=\frac{2\pi\sqrt{n}}{c}$. It follows from \eqref{phi- transform def}, \eqref{Phi-- itdef} and \eqref{S1 phi1 def} that
\begin{multline}\label{phi1-hat  int1}
\widehat{\phi_1}^{-}\left(-y^2\right)=
\frac{-y}{\sqrt{2}}\int_{0}^{\infty}\left(F_{2it}(2y\sqrt{x})+G_{2it}(2y\sqrt{x}\right)
\frac{U(x/L)}{x^{1+i\sgh(h)t}}\\\times
V\left(\frac{x+(qc)^2}{4qc},\sgh(h)t,T\right)\LE\left(\sgh(h)T,\frac{x-(qc)^2}{x+(qc)^2}\right)dx.
\end{multline}
For $n\gg\frac{C^2t^2}{L}T^{\epsilon}$, arguing as in the proof \eqref{phi2+hat  est1}, we obtain the estimate \eqref{phi1-hat<0  est1}.

Next, let us consider the remaining case: $n\ll\frac{C^2t^2}{L}T^{\epsilon}$. First, we make the change of variables $x=c^2z$ in \eqref{phi1-hat  int1}. As a result,
\begin{multline}\label{phi1-hat  int2}
\widehat{\phi_1}^{-}\left(-y^2\right)=
\frac{-y}{c^{2i\sgh(h)t}\sqrt{2}}\int_{0}^{\infty}\left(F_{2it}(4\pi\sqrt{nz})+G_{2it}(4\pi\sqrt{nz}\right)\\\times
\frac{U(c^2z/L)}{z^{1+i\sgh(h)t}}V\left(c\frac{z+q^2}{4q},\sgh(h)t,T\right)
\LE\left(\sgh(h)T,\frac{z-q^2}{z+q^2}\right)dz.
\end{multline}
For the expression above, we substitute the asymptotic formula \eqref{GF2tz asympt}. As usual, it is enough to consider only the contribution of the main term:
\begin{equation}\label{phi1-hat  int3}
\widehat{\phi_1}^{-}\left(-y^2\right)\sim
\frac{\sqrt{n}}{c^{1+2i\sgh(h)t}\sqrt{t}}\int_{0}^{\infty}U_7(z)e^{2i\sgh(h)TF_7(z,\pm)}dz,
\end{equation}
where
\begin{equation}\label{phi1-hat U7def}
U_7(z)=
\frac{U(c^2z/L)}{z\left(1+\frac{4\pi^2nz}{t^2}\right)^{1/4}}V\left(c\frac{z+q^2}{4q},\sgh(h)t,T\right)
\exp\left(-G^2\log^2\frac{x(z)+\sqrt{x^2(z)-1+\alt_0^2}}{1+\alt_0}\right),
\end{equation}
\begin{equation}\label{phi1-hat F7def}
F_7(z,\pm)=\Le(x(z))\pm\sgh(h)\alt_0\xi\left(\frac{2\pi\sqrt{nz}}{t}\right)-\alt_0\log\sqrt{z},\quad  x(z)=\frac{z-q^2}{z+q^2}.
\end{equation}
Using \eqref{GF2tz xi def} and \eqref{phi2+ h1 deriv} we obtain
\begin{equation}\label{phi1-hat F7deriv}
\frac{d}{dz}F_7(z,\pm)=\frac{-(1-\alt_0^2)}{\alt_0x(z)+\sqrt{x(z)^2-1+\alt_0^2}}\frac{2q^2}{(z+q^2)^2}
\pm\sgh(h)\frac{\alt_0}{2z}\sqrt{1+\frac{4\pi^2nz}{t^2}}-\frac{\alt_0}{2z}.
\end{equation}
Therefore, $|\frac{d}{dz}F_7(z,-\sgh(h))|\gg\frac{\alt_0}{z}$ and this case as before produces a negligible contribution. So it is required to consider
\begin{equation}\label{phi1-hat F7deriv2}
\frac{d}{dz}F_7(z,\sgh(h))=\frac{-(1-\alt_0^2)}{\alt_0x(z)+\sqrt{x(z)^2-1+\alt_0^2}}\frac{2q^2}{(z+q^2)^2}
+\frac{2\pi^2\alt_0n/t^2}{1+\sqrt{1+\frac{4\pi^2nz}{t^2}}}.
\end{equation}
The saddle point $z_1$ is a solution of the equation:
\begin{equation}\label{F7-sp}
\frac{1-\alt_0^2}{\alt_0x(z)+\sqrt{x(z)^2-1+\alt_0^2}}\frac{2q^2}{(z+q^2)^2}=\frac{2\pi^2n\alt_0/t^2}{1+\sqrt{1+\frac{4\pi^2nz}{t^2}}}.
\end{equation}
The left-hand side of \eqref{F7-sp} is approximately $q^2\alt_0^{-1}z_1^{-2}$, and as a result,
$z_1\approx\frac{qt}{\alt_0\sqrt{n}}$. This point belongs to the interval of integration only if $n\sim\frac{q^2c^4t^2}{\alt_0^2L^2}=N$ (this is the same $N$ as in \eqref{N def}). Thus similarly to Lemma \ref{lem:phi2+hat<0 est} we consider three cases: $n\ll N,$ $n\sim N$ and $n\gg N$.   Note that due to \eqref{L conditons} one has $$\frac{q^2c^4t^2}{\alt_0^2L^2}\ll\frac{c^2t^2}{2L}T^{-\epsilon}$$, and therefore, for $n<cN$ the estimate  $$\frac{4\pi^2nz}{t^2}\ll T^{-\epsilon}$$ holds.
The cases $n\ll N$ and $n\gg N$ can be treated in the same way as in Lemma \ref{lem:phi2+hat<0 est}. Doing so, we prove \eqref{phi1-hat<0  est1}. Note that now the case
$n\gg N$  is simpler than in Lemma \ref{lem:phi2+hat<0 est} due to the absence of $\chi()$ function and the plus sign under the square root in \eqref{phi1-hat F7deriv2}.
Consider the case $n\sim N$. Once again the function $U_7(z)$ can be simplified to the following expression:
\begin{equation}\label{phi1-hat U7def2}
U_7(z)=
\frac{U(c^2z/L)}{z}V\left(c\frac{z+q^2}{4q},\sgh(h)t,T\right)
\exp\left(-G^2\log^2\frac{x(z)+\sqrt{x^2(z)-1+\alt_0^2}}{1+\alt_0}\right).
\end{equation}
Consequently, (see \eqref{phi2+hat U3deriv}) the parameter $V$ in \cite[Proposition 8.2]{BKY} can be taken as  $V=L/(c^2T^{\epsilon})$.  Evaluating the derivative of \eqref{phi1-hat F7deriv2}, we show that
\begin{multline}\label{phi1-hat F7 2deriv1}
\frac{d^2}{dz^2}F_7(z,\sgh(h))=
\frac{1-\alt_0^2}{\left(\alt_0x(z)+\sqrt{x(z)^2-1+\alt_0^2}\right)^2}\left(\alt_0+\frac{x(z)}{\sqrt{x(z)^2-1+\alt_0^2}}\right)\frac{4q^4}{(z+q^2)^4}+\\
+\frac{1-\alt_0^2}{\alt_0x(z)+\sqrt{x(z)^2-1+\alt_0^2}}\frac{4q^2}{(z+q^2)^3}
-\frac{\alt_0(2\pi^2n/t^2)^2}{\left(1+\sqrt{1+\frac{4\pi^2nz}{t^2}}\right)^2\sqrt{1+\frac{4\pi^2nz}{t^2}}}.
\end{multline}
Up to some constants, the asymptotic relation holds:
\begin{equation}\label{phi2+hat F7 2deriv2}
\frac{d^2}{dz^2}F_7(z,\sgh(h))\sim
\frac{q^4}{\alt_0^3z^4}+\frac{q^2}{\alt_0z^3}-\frac{\alt_0n^2}{t^4}.
\end{equation}
Since $n\sim N$ one has $\frac{\alt_0n^2}{t^4}\sim\frac{q^4}{\alt_0^3z^4}\ll\frac{q^2}{\alt_0z^3}T^{-\epsilon}$. Therefore,
\begin{equation}\label{phi1-hat F7 2deriv3}
\frac{d^2}{dz^2}F_7(z,\sgh(h))\sim\frac{q^2}{\alt_0z^3} \quad\hbox{and}\quad
|\frac{d^j}{dz^j}F_7(z,\sgh(h))|\ll\frac{q^2}{\alt_0z^{1+j}}.
\end{equation}
Using \eqref{F7-sp} one can write down an asymptotic formula for $F_7''(z_1,\sgh(h))$ but  the main term  will still depend on $z_1$:
\begin{multline}\label{phi1-hat F7 2derivasympt}
\frac{d^2}{dz^2}F_7(z,\sgh(h))=
\frac{1-\alt_0^2}{\alt_0x(z)+\sqrt{x(z)^2-1+\alt_0^2}}\frac{4q^2}{(z+q^2)^3}\left(1+O(T^{-\epsilon})\right)=\\=
\frac{4\pi^2n\alt_0/t^2}{\left(1+\sqrt{1+\frac{4\pi^2nz}{t^2}}\right)(z_1+q^2)}\left(1+O(T^{-\epsilon})\right)=
\frac{2\pi^2n\alt_0}{z_1t^2}\left(1+O(T^{-\epsilon})\right).
\end{multline}
Finally, applying \cite[Proposition 8.2]{BKY} with parameters \eqref{BKY8.2 param}, we prove that
\begin{equation}\label{phi1-hat  int4}
\widehat{\phi_1}^{-}\left(y^2\right)\sim
\frac{\sqrt{z_1}}{c^{1+2i\sgh(h)t}}U_7(z_1)e^{2i\sgh(h)TF_7(z_1,\pm)}.
\end{equation}
\end{proof}

The proof of the next result  is similar to the one of Lemma \ref{lem^phi2+hat est}.
\begin{lem}\label{lem:phi1-hat>0 est}
Suppose that \eqref{Q0 conditions GtT2} and  \eqref{L conditons} hold. Then
\begin{equation}\label{phi1-hat>0  est1}
\widehat{\phi_1}^{-}\left(\frac{4\pi^2n}{c^2}\right)\ll
\frac{1}{(nT)^{A}},\quad\hbox{for}\quad
n\gg\frac{C^2t^2}{L}T^{\epsilon},
\end{equation}
\begin{equation}\label{phi1-hat>0  est2}
\widehat{\phi_1}^{-}\left(\frac{4\pi^2n}{c^2}\right)\ll \frac{1}{T^{A}},\quad\hbox{for}\quad
0<n\ll\frac{C^2t^2}{L}T^{\epsilon}.
\end{equation}
\end{lem}
\begin{proof}
Let $y=\frac{2\pi\sqrt{n}}{c}$. It follows from \eqref{phi- transform def}, \eqref{Phi+- itdef} and \eqref{S1 phi1 def} that
\begin{multline}\label{phi1-hat>0  int1}
\widehat{\phi_1}^{-}\left(y^2\right)=
\frac{2y}{\pi}\sin(\pi/4-\pi it)\int_{0}^{\infty}K_{2it}(2y\sqrt{x})
\frac{U(x/L)}{x^{1+i\sgh(h)t}}\\\times
V\left(\frac{x+(qc)^2}{4qc},\sgh(h)t,T\right)
\LE\left(\sgh(h)T,\frac{x-(qc)^2}{x+(qc)^2}\right)dx.
\end{multline}
For $n\gg\frac{C^2t^2}{L}T^{\epsilon}$  the integral \eqref{phi1-hat>0  int1} is bounded by $(nT)^{-A}$. The proof is similar to the one of \eqref{phi2+hat  est1},
since the differential equations \eqref{K Dun difeq} and \eqref{FG Dun difeq} differ from each other only slightly in the range $z\gg T^{\epsilon}.$
From now on, we assume that $n\ll\frac{C^2t^2}{L}T^{\epsilon}$. Making the change of variables $x=c^2z$ in \eqref{phi1-hat>0  int1}, we show that
\begin{multline}\label{phi1-hat>0  int2}
\widehat{\phi_1}^{-}\left(y^2\right)=
\frac{2y}{\pi c^{2i\sgh(h)t}}\sin(\pi/4-\pi it)\int_{0}^{\infty}K_{2it}(4\pi\sqrt{nz})
\frac{U(c^2z/L)}{z^{1+i\sgh(h)t}}\\\times
V\left(c\frac{z+q^2}{4qc},\sgh(h)t,T\right)
\LE\left(\sgh(h)T,\frac{z-q^2}{z+q^2}\right)dx.
\end{multline}
Next, similarly to Lemma \ref{lem:phi2+hat<0 est} we make a partition of the integral \eqref{phi1-hat>0  int2} with the use of \eqref{partition of 1}. The part with $\chi_{\infty}()$ is  negligible since the K-Bessel function is rapidly decreasing. In the same manner as in Lemma \ref{lem:phi2+hat<0 est} one can show that the part with $\chi_{1}()$  is also negligible. So we are left to study the part with $\chi_{0}()$.  As usual, it is enough to consider only the contribution of the main term of the asymptotic expansion \eqref{K2tz asympt}:
\begin{equation}\label{phi1-hat>0 int3}
\widehat{\phi_1}^{-}\left(y^2\right)\sim\frac{\sqrt{n}}{c^{1+2i\sgh(h)t}\sqrt{t}}
\int_{0}^{\infty}U_8(z)e^{2i\sgh(h)TF_8(z,\pm)}dz,
\end{equation}
where
\begin{equation}\label{phi1-hat>0 U8def}
U_8(z)=\chi_{0}\left(\frac{2\pi\sqrt{nz}}{t}\right)
\frac{U(c^2z/L)}{z\left(1-\frac{4\pi^2nz}{t^2}\right)^{1/4}}V\left(c\frac{z+q^2}{4q},\sgh(h)t,T\right)
\exp\left(-G^2\log^2\frac{x(z)+\sqrt{x^2(z)-1+\alt_0^2}}{1+\alt_0}\right),
\end{equation}
\begin{equation}\label{phi1-hat>0 F8def}
F_8(z,\pm)=\Le(x(z))\pm\sgh(h)\alt_0\eta\left(\frac{2\pi\sqrt{nz}}{t}\right)-\alt_0\log\sqrt{z},\quad  x(z)=\frac{z-q^2}{z+q^2}.
\end{equation}
Note that $U_8(z)\neq0$ only if $n<\frac{c^2t^2}{2\pi^2L}(1-t^{-2/3+\epsilon})^2$.
Using \eqref{K2tz eta def} and \eqref{phi2+ h1 deriv}, we obtain
\begin{equation}\label{phi1-hat>0 F8deriv}
\frac{d}{dz}F_8(z,\pm)=\frac{-(1-\alt_0^2)}{\alt_0x(z)+\sqrt{x(z)^2-1+\alt_0^2}}\frac{2q^2}{(z+q^2)^2}
\mp\sgh(h)\frac{\alt_0}{2z}\sqrt{1-\frac{4\pi^2nz}{t^2}}-\frac{\alt_0}{2z}.
\end{equation}
Since $$\frac{d}{dz}F_8(z,\sgh(h))\gg\frac{\alt_0}{2z},$$ this case is easier than the second one, namely
\begin{equation}\label{phi1-hat>0 F8deriv-}
\frac{d}{dz}F_8(z,-\sgh(h))=\frac{-(1-\alt_0^2)}{\alt_0x(z)+\sqrt{x(z)^2-1+\alt_0^2}}\frac{2q^2}{(z+q^2)^2}
-\frac{2\pi^2n\alt_0/t^2}{1+\sqrt{1-\frac{4\pi^2nz}{t^2}}}.
\end{equation}
We remark that in the last case there are no saddle points and that the derivative \eqref{phi1-hat>0 F8deriv-} has much in common with the one in
\eqref{phi2+hat F3deriv2}. Nevertheless, the presence of  $\chi_{0}\left(\frac{2\pi\sqrt{nz}}{t}\right)$ and $\left(1-\frac{4\pi^2nz}{t^2}\right)^{1/4}$ in $U_8(z)$ makes computations in some sense similar to the one performed for the functions $F_5(z),U_5(z)$ (see Lemma \ref{lem:phi2+hat<0 est}).

To estimate \eqref{phi1-hat>0 int3} we apply Lemma \ref{Lemma BKY}.
To this end, we need to determine  parameters  $P,R,V,X,Y$ in \eqref{BKYconditions}.
One has
\begin{equation}\label{phi1-hat>0 F8deriv- 2}
\left|\frac{d}{dz}F_8(z,-\sgh(h))\right|\gg
\frac{q^2}{\alt_0z^2}+\frac{\alt_0n}{t^2}.
\end{equation}
Consequently,
\begin{equation}\label{phi1-hat>0 F8 R def}
R=T\frac{q^2c^4}{\alt_0L^2}+\frac{\alt_0n}{t^2}T.
\end{equation}
Evaluating the second derivative of $F_8(z,-\sgh(h))$, we show that
\begin{equation}\label{phi1-hat F8 2deriv}
\frac{d^2}{dz^2}F_8(z,-\sgh(h))\ll\frac{1}{\alt_0^3}\frac{q^4}{z^4}+\frac{q^2}{\alt_0z^3}+\frac{\alt_0n^2}{t^4\sqrt{1-\frac{4\pi^2nz}{t^2}}}\ll
\frac{q^2}{\alt_0z^3}+\frac{\alt_0n^2}{t^4\sqrt{1-\frac{4\pi^2nz}{t^2}}}.
\end{equation}

Suppose that $\delta c^2t^2/L<n<\frac{c^2t^2}{2\pi^2L}(1-t^{-2/3+\epsilon})^2$. In this case, $$R=\frac{T\alt_0c^2}{L},
\quad P=\frac{L}{t^{2/3-\epsilon}c^2}, \quad Y=PR.$$ Evaluating the derivatives of $U_8(z)$ (see \eqref{U5 deriv}), we infer that $$V=\frac{L}{c^2t^{2/3-\epsilon}}, \quad X=t^{1/6-\epsilon}.$$ Therefore,
\begin{equation}
RP=RV>\frac{T\alt_0}{t^{2/3-\epsilon}}=t^{1/3+\epsilon},
\end{equation}
and we prove that
\begin{equation}\label{phi1-hat>0 int4}
\int_{0}^{\infty}U_8(z)e^{2i\sgh(h)TF_8(z,\pm)}dz\ll\frac{t^{1/6-\epsilon}}{(RV)^{A}}+\frac{t^{1/6-\epsilon}}{(RP)^{A}}\ll T^{-A}.
\end{equation}

Suppose that $n<\delta c^2t^2/L$ and thus $1-\frac{4\pi^2nz}{t^2}\gg1.$ In this case, $$P=L/c^2, \quad Y=PR, \quad V=\frac{L}{c^2}T^{-\epsilon}.$$

Furthermore, let us suppose that $\frac{q^2t^2c^4}{\alt_0^2L^2}<n<\delta c^2t^2/L.$ Then $R=\frac{\alt_0n}{t^2}T$. Accordingly, using \eqref{L conditons}, we obtain
\begin{equation}
RP>RV>\frac{\alt_0n}{t^2}\frac{L}{c^2}T^{1-\epsilon}>
\frac{q^2c^2}{\alt_0L}T^{1-\epsilon}\gg T^{\epsilon},
\end{equation}
which results in \eqref{phi1-hat>0 int4}.

Suppose that $n<\frac{q^2t^2c^4}{\alt_0^2L^2}.$ Then $R=T\frac{q^2c^4}{\alt_0L^2}$ and
\begin{equation}
RP>RV>\frac{q^2c^2}{\alt_0L}T^{1-\epsilon}\gg T^{\epsilon}.
\end{equation}
This proves  \eqref{phi1-hat>0 int4}.

\end{proof}


\section{Proof of Theorem \ref{BF t-result}}\label{sec:proof of t theorem}
The proof of Theorem \ref{BF t-result} is similar to the one of \cite[Theorem 1.1]{BFsym2}  presented in \cite[sec.6]{BFsym2}.
We split  the sum over $c$ in \eqref{S2(CQL) def}, \eqref{S1(CQL) def} into three parts:
\begin{equation}
c\equiv0\pmod{4},\, (c,2)=1,\, c\equiv2\pmod{4},
\end{equation}
and then apply a suitable Voronoi summation formula to each of the cases. As in \cite[sec.6]{BFsym2},  in order to demonstrate the proof it is enough to consider only one case, for example when $c\equiv0\pmod{4}$.
Applying the Voronoi formula (see Lemma \ref{Thm Voronoi+ it c0mod4}) and using the properties of integral transforms proved in Section \ref{sec:integral transforms}, we obtain
\begin{multline}\label{S2(CQL)est1}
S_2(C,Q,L)\ll
\sum_{c\equiv0\pmod{4}}U\left(\frac{c}{C}\right)\sum_{q}U\left(\frac{q}{Q}\right)
\sum_{n}\STM_0(c,-n)
\frac{\Zag_{-n}(1/2+2it)}{|n|^{1/2-it}}\\\times
\frac{U_0(z_0)e^{2iTF_0(z_0)}}{c^{1+2it}\sqrt{z_0}}U\left(\frac{n\alt_0^2L^2}{Q^2C^4t^2}\right),
\end{multline}
\begin{multline}\label{S1(CQL)est1}
S_1(h,C,Q,L)\ll
\sum_{c\equiv0\pmod{4}}U\left(\frac{c}{C}\right)\sum_{q}U\left(\frac{q}{Q}\right)
\sum_{n}
\STM_0(c,-n)\frac{\Zag_{-n}(1/2+2it)}{|n|^{1/2-it}}\\\times
\frac{U_1(z_1)e^{2i\sgh(h)TF_1(z_1)}}{c^{1+2i\sgh(h)t}\sqrt{z_1}}U\left(\frac{n\alt_0^2L^2}{Q^2C^4t^2}\right).
\end{multline}
The right-hand sides \eqref{S2(CQL)est1} and \eqref{S1(CQL)est1} are  almost identical, so we consider further only \eqref{S2(CQL)est1}. Since $F_0(z_0)$ does not depend on $c$, we first perform the summation over $c$ in the same way as in \cite[sec.6]{BFsym2}. Let $\W(c)=U\left(\frac{c}{C}\right)U_0(z_0)$, then applying \eqref{STM0 series2}, we show that
\begin{multline}\label{S2(CQL)est2}
\sum_{c\equiv0\pmod{4}}U\left(\frac{c}{C}\right)U_0(z_0)
\frac{\STM_0(c,-n)}{c^{1+2it}}=\\ =\frac{1}{2\pi i}\int_{(1/2)}\tilde{\W}(z)
\frac{L^{\ast}(1/2+2it+z,-n)}{\zeta_2(1+4it+2z)}\overline{r_2(1/2-it+\bar{z}/2,-n)}dz.
\end{multline}
Note that $\tilde{\W}(z)\ll C^{\Re{z}}((1+|\Im{z}|)T^{-\epsilon})^{-A}.$ Hence truncating the integral in \eqref{S2(CQL)est2} at $\Im{z}=T^{\epsilon}$ produces a negligibly small error term. Using $z_0^{-1/2}\approx C/\sqrt{L}\approx\frac{\sqrt{nL}}{TCQ}$, \eqref{Last 2mom estimate} and \eqref{LZag 2mom estimate}, we prove that
\begin{equation}\label{S2(CQL)est3}
S_2(C,Q,L)\ll
Q\frac{\sqrt{L}}{TCQ}\left(\frac{Q^2C^4t^2}{\alt_0^2L^2}+\frac{QC^2t^{3/2}}{\alt_0L}\right)T^{\epsilon}\ll
\left(\frac{Q^2C^3T}{L^{3/2}}+\frac{QCt^{1/2}}{L^{1/2}}\right)T^{\epsilon}.
\end{equation}
Substituting \eqref{S2(CQL)est3} to \eqref{S2 to S2(CQL)} yields
\begin{multline}\label{S2 to S2(CQL) est1}
S_2(t)\ll \frac{T^{1+\epsilon}G}{t^{1/2}}
\sum_{C\ll t^{3/2}G^{-1+\epsilon}}\sum_{Q_0/C\ll Q\ll t^{3/2}G^{-1+\epsilon}/C}
\sum_{Q^2C^2G^{1-\epsilon}/\alt_0\ll L\ll QCT^{1+\epsilon}\sqrt{t}}
\left(
\frac{Q^2C^3T}{L^{3/2}}+\frac{QCt^{1/2}}{L^{1/2}}\right)\ll\\
\frac{T^{1+\epsilon}G}{t^{1/2}}\sum_{C\ll t^{3/2}G^{-1+\epsilon}}\sum_{Q_0/C\ll Q\ll t^{3/2}G^{-1+\epsilon}/C}
\left(\frac{T\alt_0^{3/2}}{G^{3/2}Q}+\frac{t^{1/2}\alt_0^{1/2}}{G^{1/2}}\right)\ll\\\ll
T^{1+\epsilon}G\left(\frac{t}{G^{3/2}T^{1/2}}+\frac{t^{1/2}}{G^{1/2}T^{1/2}}\right)\ll T^{1-\epsilon}G,
\end{multline}
provided that
\begin{equation}\label{G condition}
G\gg \frac{t^{2/3}}{T^{1/3-\epsilon}}.
\end{equation}
The same estimate holds for $S_1(t)$, and therefore, (see \eqref{t2mom est S1+S2}) we prove \eqref{BF t-mean Lindelof} provided that
\eqref{Q0 conditions GtT2} is satisfied, that is
\begin{equation}\label{G conditions final}
G\gg \max\left(\frac{t^{2/3}}{T^{1/3}},\frac{T}{t}\right)T^{\epsilon},\quad  t\ll T^{6/7-\epsilon}.
\end{equation}


\end{document}